\documentclass[11pt,reqno]{article}
\usepackage{amsmath,amsfonts,amsthm}
\usepackage[margin=1in]{geometry}
\usepackage[normalem]{ulem}
\usepackage{ amssymb }

\usepackage[hyphens]{url}
\usepackage{hyperref}
\usepackage{graphicx}
\usepackage{verbatim}
\usepackage{enumitem} 
\usepackage{color}
\usepackage{tikz}
\usetikzlibrary{decorations.fractals}
\usetikzlibrary{patterns,arrows,shapes,decorations.pathreplacing}
\usepackage{bm}

\numberwithin{equation}{section}

\newtheorem{theorem}{Theorem}[section]

\newtheorem{corollary}{Corollary}[section]
\newtheorem{lemma}{Lemma}[section]
\newtheorem{proposition}{Proposition}[section]

\newtheorem{remark}{Remark}[section]

\renewcommand{\P}{\mathbb{P}}

\newcommand{\R}{\mathbb{R}}
\newcommand{\E}{\mathbb{E}}

\newcommand{\N}{\mathbb{N}}
\newcommand{\F}{\mathcal{F}}

\newcommand{\T}{\mathcal{T}}

\newcommand{\C}{\mathcal{C}}

\newcommand{\eps}{\varepsilon}

\newcommand{\nada}[1]{}
\definecolor{gb}{rgb}{0, 0.2, 0.8}

\title{Optimal Consumption in the Stochastic Ramsey Problem\\ without Boundedness Constraints
}
\author{Yu-Jui Huang\thanks{
University of Colorado, Department of Applied Mathematics, Boulder, CO 80309-0526, USA, email: \texttt{yujui.huang@colorado.edu}. Partially supported by National Science Foundation (DMS-1715439) and the University of Colorado (11003573).}
 \and Saeed Khalili\thanks{
University of Colorado, Department of Mathematics, Boulder, CO 80309-0395, USA, email: \texttt{saeed.khalili@colorado.edu}.}
}
\date{\today}
\begin{document}
\maketitle

\begin{abstract}
This paper investigates optimal consumption in the stochastic Ramsey problem with the Cobb-Douglas production function. Contrary to prior studies, we allow for general consumption processes, without any a priori boundedness constraint. A non-standard stochastic differential equation, with neither Lipschitz continuity nor linear growth, specifies the dynamics of the controlled state process. A mixture of probabilistic arguments are used to construct the state process, and establish its non-explosiveness and strict positivity. This leads to the optimality of a feedback consumption process, defined in terms of the value function and the state process. Based on additional viscosity solutions techniques, we characterize the value function as the unique classical solution to a nonlinear elliptic equation, among an appropriate class of functions. This characterization involves a condition on the limiting behavior of the value function at the origin, which is the key to dealing with unbounded consumptions. Finally, relaxing the boundedness constraint is shown to increase, strictly, the expected utility at all wealth levels. 
\end{abstract}

\smallskip



\section{Introduction}
In the economic growth theory, capital stock of a society amounts to the total value of assets that can be used to produce goods and services, such as factories, equipment, and monetary resources. Whereas capital can be consumed to give individuals immediate welfare, it can also be used to generate more capital and thus sustain economic growth, which enhances future welfare. As Ramsey \cite{Ramsey28} pointed out in a deterministic model, sensible financial planning, regarding consumption and saving of capital, is imperative to strike a balance between current and future welfare. In a continuous-time setting, Merton \cite{Merton75} enriched the problem by considering stochastic evolution of the population in a society. 

The stochastic Ramsey problem, coined by Merton \cite{Merton75}, has been investigated in the stochastic control literature through viscosity solution techniques, Banach's fixed-point argument, and the combination of both; see e.g. Morimoto and Zhou \cite{MZ08}, Morimoto \cite{Morimoto08, Morimoto-book-10}, and Liu \cite{Liu14}, among others. Surprisingly, many of these works require an {\it a priori} uniform upper bound, usually the constant $1$, for consumption processes $\{c_t\}_{t\ge 0}$. This is implicitly suggested in the problem formulation of \cite{Merton75}, and explicitly stated as $0\le c_t\le 1$ in \cite{MZ08} and \cite{Liu14}. While this uniform upper bound provides technical conveniences, 
it can {\it not} be fully justified economically in continuous time. After all, for each $t\ge 0$, $c_t$ represents the consumption ratio {\it per unit of time} instantly at time $t$, which does not admit any natural upper bound. This is in contrast to the discrete-time setting where the upper bound 1 can be easily justified. Morimoto \cite{Morimoto08, Morimoto-book-10} consider general, unbounded consumption processes, but not without a cost. There, the production function in the Ramsey model is required to have finite first derivatives, along a boundary of its domain. This particularly rules out the standard Cobb-Douglas production function, commonly used in economic modeling. 

In other words, a tradeoff exists between the viscosity solutions approach in \cite{MZ08, Liu14} and Banach's fixed-point argument in \cite{Morimoto08, Morimoto-book-10}. The former accommodates the classical Cobb-Douglas production function, but is limited to uniformly bounded consumption processes; the latter allows for general consumptions,  but fails to cover the Cobb-Douglas production function. We aim to resolve this tradeoff: this paper considers {\it both} unbounded consumption processes and the Cobb-Douglas production function, in the stochastic Ramsey problem. The goal is to characterize the associated value function $V$, as well as a (possibly unbounded) optimal consumption process $\hat c$. 


The upfront challenge of our studies is the non-standard stochastic differential equation (SDE) of the state process $X$, which represents {\it capital per capita}; see \eqref{X} below. On the one hand, the Cobb-Douglas production function renders the drift coefficient of $X$ non-Lipschitz (see Section~\ref{subsec:comparison} for a comparison with the Lipschitz case \cite{Morimoto08, Morimoto-book-10}). On the other hand, the unboundedness of consumptions may induce superlinear growth in the same drift coefficient, in contrast to \cite{MZ08, Liu14} where linear growth is guaranteed (see Remark~\ref{rem:linear growth in MZ08}). With {\it neither} Lipschitz {\it nor} linear growth condition, standard techniques for SDEs cannot be applied. Instead, we investigate the existence and uniqueness of $X$, by constructing solutions directly. In Proposition~\ref{prop:existence} and Corollary~\ref{coro:X=0}, we establish the existence of $X$, yet observe that the uniqueness fails in general. Based on the construction of $X$, we also derive moment estimates in Proposition~\ref{prop:estimates X}, without resorting to linear growth condition.

With the state process $X$ constructed, we proceed to relate our value function $V$ to a differential equation. Our strategy is to approximate $V$ by $V_L$, the value function when one is restricted to consumption processes uniformly bounded by $L>0$. By generalizing arguments in \cite{MZ08} to infinite horizon, $V_L$ is shown to be a classical solution to a nonlinear elliptic equation (Proposition~\ref{prop:properties V_L}). As $L\to\infty$, we prove that $V_L$ converges to $V$ desirably, such that $V$ is a classical solution to the limiting nonlinear elliptic equation (Proposition~\ref{prop:classical V_infty} and Theorem~\ref{thm:V=V_infty}).

There are two remaining tasks: (i) to find an optimal consumption process $\hat c$, and (ii) to characterize $V$ further as the {\it unique} classical solution among a certain class of functions.

While $\hat c$ can be heuristically derived in feedback form (i.e. $\hat c_t = \hat c(X_t)$), it is highly nontrivial whether the controlled state process $X^{\hat c}$ is well-defined. First, whether $X^{\hat c}$ exists is unclear: The aforementioned existence result of $X$ does not apply here, as the current control process $\hat c$ is not {\it a priori} given, but depends on the unknown $X$. Second, even if $X^{\hat c}$ exists, it is in question whether the dire situation ``$X^{\hat c}_t =0$ for some $t> 0$'' (i.e. the society using up all its capital at time $t$) can be avoided. A careful construction of $X^{\hat c}$, along with a detailed analysis on its explosion and pathwise uniqueness, is carried out in Proposition~\ref{prop:X'}. It shows that $X^{\hat c}$ is indeed a well-defined strictly positive process, on the strength of Feller's test for explosion and a mixture of probabilistic arguments in Nakao \cite{Nakao72} and Yamada \cite{Yamada73}. Now, with $X^{\hat c}$ well-defined and $V$ solving a nonlinear elliptic equation, a standard verification argument establishes the optimality of $\hat c$. 

Note that the construction of $X^{\hat c}$ was done with much more ease in \cite{MZ08}, through a change of measure. 
This works, however, only with bounded consumptions and finite time horizon. That is, Proposition~\ref{prop:X'} complements \cite{MZ08}, by providing a new, different construction that is general enough to accommodate both unbounded consumptions and infinite horizon; see Remark~\ref{rem:Girsanov} for details.

In fact, the construction in Proposition~\ref{prop:X'} can be made much more general. For any $u\in C^1((0,\infty))$ that is strictly increasing, concave, and whose behavior at $0+$ satisfies \eqref{C*/x^alpha'} below, we can construct from $u$ a candidate optimal consumption $\hat c^u$, and show that the state process $X^{\hat c^u}$ is well-defined and strictly positive; see Corollary~\ref{coro:X' with u} and \eqref{hat c^u}. With the aid of a verification argument, this leads to the {\it full} characterization: $V$ is the unique classical solution to a nonlinear elliptic equation among the class of functions $u\in C^2((0,\infty))\cap C([0,\infty))$ that are strictly increasing, concave, satisfying \eqref{C*/x^alpha'} and the linear growth condition; see Theorem~\ref{thm:main}. 

In \cite{MZ08}, where consumptions are uniformly bounded, the value function is only shown to be a classical solution, with no further characterization. Theorem~\ref{thm:main} fills this void, in a more general setting with unbounded consumptions; see Remark~\ref{rem:fills void}. Specifically, the identification of \eqref{C*/x^alpha'} in Theorem~\ref{thm:main} is the key to dealing with unbounded consumptions. If one is restricted to bounded consumptions (as in \cite{MZ08}), there is no need to impose \eqref{C*/x^alpha'}; see Remark~\ref{rem:characterize V_L}. 


Finally, we compare our no-constraint optimal consumption $\hat c$ with the optimal $\hat c^L$ in \cite{MZ08}, bounded by $L>0$. Two questions are particularly of interest. First, by switching from the bounded strategy $\hat c^L$ to the possibly unbounded $\hat c$, can we truly increase our expected utility? An affirmative answer is provided in Proposition~\ref{prop:V>V_L}: expected utility rises at {\it all} levels of wealth (capital per capita), whenever $\hat c$ is truly unbounded. This justifies economically the use of unbounded strategies. Second, for each $L>0$, do agents following $\hat c^L$ simply chop the no-constraint optimal strategy $\hat c$ at the bound $L>0$? Corollary~\ref{coro:statement not true} shows that the relation ``$\hat c^L = \hat c \wedge L$'' fails in general, suggesting a more structural change from $\hat c^L$ to $\hat c$. For the isoelastic utility function $U(x)=\frac{x^{1-\gamma}}{1-\gamma}$, $0<\gamma<1$, we demonstrate the above two results fairly explicitly.  

The paper is organized as follows. Section~\ref{sec:model} introduces the stochastic Ramsey problem with general unbounded consumptions. Section~\ref{sec:X} investigates the existence and uniqueness of the state process $X$, and derives moment estimates of it. Section~\ref{sec:V_L} shows that the value function $V$ is a classical solution to a nonlinear elliptic equation. Section~\ref{sec:optimal consumption} finds an optimal consumption $\hat c$, and establishes a full characterization of $V$. Section~\ref{sec:comparison} compares our results with previous literature with bounded consumptions. Appendix~\ref{sec:appendix} generalizes arguments in \cite{MZ08} to infinite horizon.


\section{The Model}\label{sec:model}
Consider the canonical space $\Omega := \{\omega\in C([0,\infty);\R)\mid \omega_0=0\}$ of continuous paths starting with value $0$. Let $W$ be the canonical process on $\Omega$, $\P$ be the Wiener measure, and $\mathbb{F}= \{\F_t\}_{t\ge 0}$ be the $\P$-augmentation of the natural filtration generated by $W$. Given $t>0$ and $\omega\in\Omega$, for any $\bar\omega\in\Omega$, we define the concatenation of $\omega$ and $\bar\omega$ at time $t$ as
\begin{equation}\label{concatenate}
(\omega\otimes_t\bar\omega)_r := \omega_r 1_{[0,t]}(r) + (\bar\omega_{r-t}+\omega_t)1_{(t,\infty)}(r),\quad  r\ge 0.
\end{equation}   
Note that $\omega\otimes_t\bar\omega$ again belongs to $\Omega$.

Consider a society in which the labor supply is equal to total population. The capital stock $K$ of the society accumulates from economic output, generated by the capital itself and the labor force. At the same time, $K$ may decrease due to capital depreciation and consumption from the population. Specifically, we assume that $K$ follows the dynamics
\[
d K_t = [F(K_t, Y_t) - \lambda K_t  - c_t K_t] dt\quad \hbox{for}\ t> 0,\qquad K_0 = k>0. 
\]
Here, $F:[0,\infty)\times[0,\infty)\to [0,\infty)$ is a production function, $Y$ is the labor supply process,  $\lambda\ge 0$ is the constant rate of depreciation, and $c$ is the consumption rate process chosen by the population. Throughout this paper, we take $F$ to be the Cobb-Douglas form, i.e. 
\begin{equation}\label{CD}
F(k,y) := k^\alpha y^{1-\alpha},\quad \hbox{for some $\alpha\in(0,1)$}. 
\end{equation}
Also, we assume that the labor supply process $Y$ is stochastic, modeled as a geometric Brownian motion:
\[
dY_t = n Y_t dt + \sigma Y_t dW_t\quad \hbox{for}\ t>0,\qquad  Y_0 = y>0.
\] 
where $n\in\R$ and $\sigma>0$ are two given constants. In addition, we consider general consumption processes $c$ without any a priori boundedness condition, as opposed to most previous studies in the literature. Specifically, the set $\C$ of admissible consumption processes is taken as 
\begin{equation}\label{C}
\C := \left\{c:\Omega\times [0,\infty)\to \R_+: c\ \hbox{is progressively measurable},\ \int_0^t c_s ds <\infty\ \ \forall t>0\ \ \hbox{a.s.}\right\}.
\end{equation}

At each time $t\ge 0$, every individual is allotted the capital $K_t/Y_t$, which can be consumed immediately or saved for future production. An individual is then faced with an optimal consumption problem: he/she intends to choose an appropriate consumption process $\hat c\in\C$, so that the expected discounted utility from consumption can be maximized. Specifically, the corresponding value function is given by
\begin{equation}\label{v}
v(k,y) := \sup_{c\in\C}\E\bigg[\int_{0}^{\infty} e^{-\beta t} U\left(c_t \frac{K^k_t}{Y^y_t}\right) dt\bigg],
\end{equation}
where $\beta\ge 0$ is the discount rate and $U:[0,\infty)\to\R$ is a utility function. We will assume that 
\begin{align}
&\hbox{$U$ is strictly increasing and strictly concave}, \label{U1}\\
&U'(0+) = U(\infty) = \infty\quad \hbox{and}\quad U'(\infty) = U(0) = 0. \label{U2}
\end{align} 


The dimension of the problem can be reduced, by introducing the variable $x := k/y$ and the process $X_t:=K_t/Y_t$, i.e. the {\it capital per capita} process. Specifically, the value function in \eqref{v} can be re-written as
\begin{equation}\label{V}
V(x):=\sup_{c\in\C}\E\bigg[\int_{0}^{\infty} e^{-\beta t} U\left(c_t X^x_t \right) dt\bigg],
\end{equation}
where the process $X$ satisfies, thanks to It\^{o}'s formula, 
\begin{align}\label{X}
dX_t = (X_t^\alpha  - \mu X_t - c_tX_t)dt - \sigma X_tdW_t  \quad t>0,\qquad  X_0 = x\ge 0,  
\end{align}
with $\mu := \lambda + n - \sigma^2$. As in \cite{MZ08}, we will assume throughout the paper that
\begin{equation}\label{mu>0}
\mu >0.
\end{equation}

The goal of this paper is to provide characterizations for the value function $V$ in \eqref{V}, as well as the associated optimal consumption process $\hat c$. 
 


\section{The Capital per Capita Process}\label{sec:X}
In this section, we analyze the capital per capita process $X$, formulated as the stochastic differential equation (SDE) \eqref{X}. We will investigate the existence and uniqueness of solutions to \eqref{X}, and derive several moment estimates for $X$, useful in Sections~\ref{sec:V_L} and \ref{sec:optimal consumption} for characterizing $V$ in \eqref{V}. 

The SDE \eqref{X} is non-standard: the drift coefficient is {\it neither} Lipschitz {\it nor} of linear growth. Indeed, Lipschitz continuity fails due to the term $X^\alpha_t$, and the unboundedness of $c$ may lead to superlinear growth. Consequently, standard techniques to establish existence and uniqueness of solutions (requiring both ``Lipschitz'' and ``linear growth'') and to derive moment estimates (requiring ``linear growth'') cannot be applied here.  

\begin{remark}\label{rem:linear growth in MZ08}
In \cite{MZ08}, \eqref{X} is studied in a simpler setting, where $c$ is assumed to be uniformly bounded (in fact, $c_t \le 1$ for all $t\ge 0$). This ensures linear growth of the drift coefficient of \eqref{X}, such that some standard techniques and estimates can still be used. 
\end{remark}


Without the aid of standard results, we investigate existence and uniqueness of solutions to \eqref{X}, by constructing solutions directly.
As shown in Proposition~\ref{prop:existence} and Corollary~\ref{coro:X=0} below, existence can be established in general, yet uniqueness need not always hold. 



\begin{proposition}\label{prop:existence}
For any $c\in\C$ and $x>0$, there exists a unique strong solution to \eqref{X}, which is strictly positive a.s.
\end{proposition}

\begin{proof}
Fix $c\in\C$ and $x>0$. Consider $ Z_t := X_t^{1 - \alpha}$, with $Z_0 = z := x^{1-\alpha}>0$. Since the function $f(y) := y^{1-\alpha}$ is well-defined on $[0,\infty)$ and differentiable on $(0,\infty)$, we can apply It\^{o}'s formula to $Z$ only up to the stopping time 
\[
\tau := \inf\{t\ge 0 : X^x_t =0\} = \inf\{t\ge 0 : Z^z_t =0\}.
\]
This gives the dynamics of $Z$ up to time $\tau$: 
\begin{equation}\label{Z}
dZ_t =  (1 - \alpha )  \left(1 -  (\mu  + c_t + \frac{1}{2}\sigma^2  \alpha  ) Z_t \right) dt  - \sigma( 1 - \alpha) Z_t  dW_t,\quad \hbox{for}\ 0<t<\tau.
\end{equation}
We claim that this SDE admits a unique strong solution. For simplicity, let $a := 1 - \alpha$ and $b_t := - (1 - \alpha) (\mu  + c_t + \frac{1}{2}\sigma^2  \alpha  )$, and define 
\begin{equation}\label{G}
G_t : = \exp\left(\int_{0}^{t}(-b_s+\frac{\sigma^2 a^2}{2})ds +\sigma a W_t\right) >0,\qquad \hbox{for}\ t\ge 0.
\end{equation}
Note that $G$ is well-defined a.s. thanks to $c\in\C$; recall \eqref{C}. By definition, $G$ satisfies the dynamics $dG_t = (-b_t+\sigma^2 a^2) G_t dt +\sigma a G_t dW_t$, for all $t>0$. By applying It\^{o}'s formula to the product process $G Z$ up to time $\tau$, we get 
\begin{align*}
d(G_t Z_t) &= G_t(a+b_tZ_t)dt -\sigma a G_t Z_t dW_t +G_t Z_t(-b_t+{\sigma^2 a^2}) dt +\sigma a G_t Z_tdW_t  -\sigma^2 a^2 G_t Z_t dt \\ 
&= aG_t dt,\quad \hbox{for}\ 0<t<\tau.
\end{align*}
This implies that
\begin{equation}\label{Z formula}
Z_t = \frac{1}{G_t} \bigg(z + (1-\alpha)\int_{0}^{t} G_s ds\bigg)
\end{equation}
is the unique strong solution to \eqref{Z}, given that $Z_0 =z$. Now, in view of \eqref{Z formula} and $G_t>0$ for all $t\ge 0$, we conclude that $Z_t >0$ for all $t\ge 0$ a.s., and thus $\tau=\infty$ a.s. 

With $\tau=\infty$ a.s., the construction in the proof above implies that the process $X_t = Z_t^{1/(1-\alpha)}$, $t\ge 0$, with $Z$ given by \eqref{Z formula}, is the unique strong solution to \eqref{X}, and it is strictly positive a.s. 
\end{proof}


For the case $x=0$ in \eqref{X}, uniqueness of solutions fails. 

\begin{corollary}\label{coro:X=0}
For any $c\in \C$, if $x =0$ in \eqref{X}, then $X\equiv 0$ and
\[
\widetilde{X}_t := 
\begin{cases}
0 \quad &\hbox{if}\ t=0,\\
\left(\frac{1-\alpha}{G_t}\int_0^t G_s ds\right)^{\frac{1}{1-\alpha}}>0\quad &\hbox{if}\ t>0,
\end{cases}
\]
are two dinstinct strong solutions to \eqref{X}. Here, $G$ is defined as in \eqref{G}.
\end{corollary}

\begin{proof}
Since $X\equiv 0$ trivially solves \eqref{X}, we focus on showing that $\widetilde{X}$ is a strong solution to \eqref{X}. First, since $G_0=1\neq 0$, $\widetilde X$ is continuous at $t=0$, i.e. $\lim_{t\downarrow 0} \widetilde{X}_t =  0 = \widetilde{X}_0$. Now, consider the SDE \eqref{Z}, with $Z_0 =0$. Due to the term $(1-\alpha) dt$, $Z$ will immediately go up from $0$, such that $\tau' := \inf\{t>0 : Z^0_t = 0\} >0$. We can then apply It\^{o}'s formula to the process $G Z$ over the interval $(0,\tau')$. Similarly to the proof of Proposition~\ref{prop:existence}, we find that $Z_t = \frac{1-\alpha}{G_t} \int_{0}^{t} G_s ds$ is the unique strong solution to \eqref{Z} up to time $\tau'$, given that $Z_0=0$. But the formula of $Z$ entails $Z_t >0$ for all $t>0$ a.s., and thus $\tau' =\infty$ a.s. Observe that $\widetilde{X}_t = (Z_t)^{1/(1-\alpha)}$ for all $t\ge 0$. With $\tau' =\infty$ a.s., we can apply It\^{o}'s formula to $\widetilde X$ over $(0,\infty)$, which shows that it is a strong solution to \eqref{X}.
\end{proof}

\begin{remark}\label{rem:V(0)=0}
Recall $V$ in \eqref{V}. By Corollary~\ref{coro:X=0},  $V(0)$ is not well-defined. Indeed, one has $V(0) = 0$ with $X\equiv 0$ in \eqref{V}, but $V(0)>0$ with $X= \widetilde X$ in \eqref{V}. 
\end{remark}

\begin{remark}
According to the boundary classification in Karlin and Taylor \cite[Chapter 15]{KT-book-81}, $x=0$ is an ``entrance boundary'' of the state space $[0,\infty)$ of $\widetilde X$ in Corolloary~\ref{coro:X=0}: beginning at the boundary $x=0$, $\widetilde X$ quickly moves to the interior and never returns to the boundary.
\end{remark}


Classical moment estimates of SDEs rely on linear growth of coefficients, along with an application of Gronwall's lemma; see e.g. Krylov \cite[Chapter 2]{Krylov-book-80}, especially Corollary 2.5.12. As mentioned before, the drift coefficient of \eqref{X} does not necessarily have linear growth, unless $c$ is known {\it a priori} a bounded process (as in \cite{MZ08}). The explicit formula of $X$ via \eqref{Z formula} turns out to be handy here. Detailed analysis on such a formula yields desirable moment estimates, without requiring any linear growth condition.


\begin{proposition}\label{prop:estimates X}
Let $\eta := \frac{1}{1-\alpha}$. Given $c\in\C$, the unique strong solution $X$ of \eqref{X} satisfies 
 \begin{align} 
 \E[X^x_t]&\leq 2^{\eta -1} ( x + t^{\eta} )\quad\hbox{and}\quad \E[(X^x_t)^2] \le 2^{2\eta-1} e^{\sigma^2 t} \left(x^2 + \frac{t^{2\eta-1}}{\sigma^2}\right),\qquad \forall x>0\ \hbox{and}\ t\ge 0.   \label{eq2.5} 
 \end{align}
Moreover, for any $\eps>0$, there exists $C_\eps >0$ such that 
\begin{equation}
\E[|X^x_t - X^y_t | ] \leq C_{\eps} |x-y| + \eps ( x + y + t^{\eta}) \qquad \forall x,y >0. \label{eq2.7}
\end{equation}
\end{proposition}

\begin{proof}
Fix $c\in\C$ and $x>0$. Consider $ Z_t := (X^x_t)  ^  {1 - \alpha} $. Then, as shown in the proof of Proposition~\ref{prop:existence}, $Z$ satisfies \eqref{Z}, which can be solved to get the formula \eqref{Z formula}. It follows that
\begin{equation}\label{X in Z}
X_t = Z^\eta_t = G_t^{-\eta} \bigg(x^{1-\alpha} + (1-\alpha)\int_{0}^{t} G_s ds\bigg)^\eta\le  2^{\eta-1} G_t^{-\eta} \bigg[x + (1-\alpha)^{\eta}\left(\int_{0}^{t} G_s ds\right)^{\eta}\bigg],
\end{equation}
where the inequality follows from $(u+v)^{k} \leq 2^{k-1}(u^{k}+v^{k})$ for $u,v\ge 0$ and $k> 1$. Observe from \eqref{G} that
\begin{equation}\label{G formula}
G_t = \exp\left(\int_{0}^{t}(1-\alpha)(\mu+ c_t+\frac{\sigma^2}{2})ds + (1-\alpha)\sigma W_t\right) >0,\qquad t\ge 0.
\end{equation}
This, together with $c_t \ge 0$, implies that 
\begin{equation}\label{E[G_t]}
\E[G_t^{-\eta}] \le \E\bigg[\exp\left((-\mu-\frac{\sigma^2}{2})t - \sigma W_t\right) \bigg] = e^{-\mu t}<1. 
\end{equation}
Now, for any $0\le s\le t$, we introduce 
\begin{equation}\label{G_s,t}
G_{s,t} := \exp\left(\int_{s}^{t}(1-\alpha)(\mu+ c_r+\frac{\sigma^2}{2})dr + (1-\alpha)\sigma (W_t-W_s)\right) >0.
\end{equation}
Then, observe that
\begin{align*}
\E\bigg[G_t^{-\eta}\left(\int_0^t G_s ds\right)^\eta\bigg] = \E\bigg[\left(\int_0^t G_{s,t}^{-1} ds\right)^\eta\bigg].
\end{align*}
By applying Jensen's inequality 
to $\big(\int_0^t G_{s,t}^{-1} ds\big)^\eta$, 
we deduce from the above equality that
\begin{equation}\label{<t^eta}
\E\bigg[G_t^{-\eta}\left(\int_0^t G_s ds\right)^\eta\bigg] \le \E\bigg[t^{\eta-1}\int_0^t G_{s,t}^{-\eta} ds \bigg] = t^{\eta-1}\int_0^t \E[G^{-\eta}_{s,t}] ds \le t^\eta,
\end{equation}
where the last inequality follows from $\E[G^{-\eta}_{s,t}]\le 1$, which can be proved as in \eqref{E[G_t]}. Now, by \eqref{E[G_t]} and \eqref{<t^eta}, we conclude from \eqref{X in Z} that $\E[X_t] \le 2^{\eta-1}(x+t^\eta)$, as desired. 
To prove the second part of \eqref{eq2.5}, we replace $\eta$ by $2\eta$ in the above arguments. First, \eqref{E[G_t]} becomes 
\begin{equation}\label{E[G_t]'}
\E[G_t^{-2 \eta}] \le \E\bigg[\exp\left((-2\mu-{\sigma^2})t - 2\sigma W_t\right) \bigg] = e^{-(2\mu-\sigma^2) t}\le e^{\sigma^2 t}. 
\end{equation}
Then, \eqref{<t^eta} becomes 
\begin{align}\label{<t^eta'}
\E\bigg[G_t^{-2\eta}\left(\int_0^t G_s ds\right)^{2\eta}\bigg] &= \E\bigg[\left(\int_0^t G_{s,t}^{-1} ds\right)^{2\eta}\bigg]\notag\\
&\le \E\bigg[t^{2\eta-1}\int_0^t G_{s,t}^{-2\eta} ds \bigg] = t^{2\eta-1}\int_0^t \E[G^{-2\eta}_{s,t}] ds \le \frac{t^{2\eta-1}}{\sigma^2}(e^{\sigma^2 t} -1),
\end{align}
where the first inequality follows from applying Jensen's inequality to $\big(\int_0^t G_{s,t}^{-1} ds\big)^{2\eta}$ 
 and the second inequality is due to $\E[G^{-2\eta}_{s,t}]\le e^{\sigma^2 (t-s)}$, which can be proved as in \eqref{E[G_t]'}. Finally, using the same calculation in \eqref{X in Z} with $\eta$ replaced by $2\eta$, along with \eqref{E[G_t]'} and \eqref{<t^eta'}, we conclude that $\E[(X^x_t)^2] \le 2^{2\eta-1} e^{\sigma^2 t} (x^2 + t^{2\eta-1}/\sigma^2)$, as desired.

To prove \eqref{eq2.7},  
consider the process $Z$ defined above, as well as $ \bar Z_t := (X^y_t) ^ {1-\alpha}$. As above, $Z$ and $\bar Z$ take the form \eqref{Z formula}, with initial values $z = x^{1-\alpha}$ and $\bar z = y^{1-\alpha}$, respectively. Thus, by \eqref{E[G_t]},  
\begin{align}\label{eq2.11}
\E[|Z_t - \bar Z_t |^\eta] \leq  |z - \bar z |^\eta\ \E[G_t^{-\eta}] \le |z - \bar z |^\eta = |x^ {1-\alpha} - y^ {1-\alpha} | ^{\frac{1} {1- \alpha}} \leq  |x-y|, 
\end{align}
where the last inequality follows from the observation $|u^ {r} - v^ {r} | \leq  |u-v|^r$ for any $u, v\ge 0$ and $0<r<1$. Indeed, we may assume without loss of generality that $u\ge v$ and define $\lambda:=u/v \ge 1$. Thus, the observation is equivalent to $\lambda^r-1 \le (\lambda-1)^r$ for any $\lambda \ge 1$ and $0<r<1$. The latter is true because $f(\lambda):=(\lambda-1)^r-\lambda^r+1$ satisfies $f(1)=0$ and $f'(\lambda) = r\big((\frac{1}{\lambda-1})^{1-r}-(\frac{1}{\lambda})^{1-r}\big)>0$ for all $\lambda> 1$. 

Next, for any $a, b \geq 0$ and $\eps>0$, observe that 
\begin{align} \label{eq2.12}
|a^{\eta} - b^{\eta}| & =\left|\int_a^b \eta r^{\eta-1} dr \right|\leq \eta |a-b| (a^{\eta-1} + b^{\eta-1}) \nonumber \\
&\leq   \frac {1}{ \eps^ {\eta } }  \ |a - b |^{\eta} + (\eta -1) \ \eps^ { \frac{ \eta  }{\eta - 1} } (a^{\eta-1} + b^{\eta-1})^{\frac{\eta }{\eta-1} } \nonumber 
\\& \leq \frac {1}{ \eps^ {\eta } }  \ |a - b |^{\eta} + (\eta -1) \ (2\eps)^ { \frac{ \eta  }{\eta - 1} } (a^{\eta} + b^{\eta}), 
\end{align}
where the second line follows from Young's inequality with $p = \eta $ and $q = \frac{\eta}{\eta -1 }$, and the third line is due to $(u+v)^{k} \leq 2^{k-1}(u^{k}+v^{k})$ for $u,v\ge 0$ and $k> 1$. 
Now, for any $\eps>0$, 
\begin{align*}
\E[|X^x_t - X^y_t | ] = \E[|Z_t ^\eta  - \bar Z_t ^ \eta| ]  &\leq  \frac{1}{\eps^{\eta } } |x-y| + (\eta-1) (2\eps)^{\frac{\eta}{\eta-1}} ( \E[ Z_t ^ {\eta} ]+ \E[\bar Z_t^{\eta} ]) \\
&\leq  \frac{1}{\eps^{\eta } }  |x-y| + 2^{\eta }(\eta-1) (2\eps)^{\frac{\eta}{\eta-1}} (  x  +  y + t^{\eta}  ),
\end{align*}
where the first inequality follows from \eqref{eq2.12} and \eqref{eq2.11}, and the second inequality is due to the first part of \eqref{eq2.5}. Now, in the last line of the previous inequality, by taking $\eps' := 2^{\eta }(\eta-1) (2\eps)^{\frac{\eta}{\eta-1}}$ and $C_{\eps'} := \frac{1}{\eps^{\eta } } = 2^{\eta^2}\big(\frac{\eta-1}{\eps'}\big)^{\eta-1}$, we see that \eqref{eq2.7} holds.
\end{proof}


\section{Properties of the Value Function}\label{sec:V_L}
In this section, we introduce, for each $L>0$, the auxiliary value function
\begin{equation}\label{V_L}
V_L(x) := \sup_{c\in\mathcal C_L}\E\left[\int _{0}^{\infty}  e^{-\beta t} U (c_t X^x_t) dt\right]\quad x\ge 0,
\end{equation}
where 
\begin{equation}\label{C_L}
\mathcal{C}_L := \{c\in\C : c_t\le L\ \hbox{for all}\ t\ge 0\}.
\end{equation}
We will first derive useful properties of $V_L$. 
As $L\to \infty$, we will see that $V_L$ converges desirably to $V$ in \eqref{V}, so that $V$ inherits many properties of $V_L$.     

Morimoto and Zhou \cite{MZ08} studied a similar problem to $V_L$: they took $L=1$ and the time horizon to be finite in \eqref{V_L}. Extending their arguments to infinite horizon gives properties of $V_L$ as below. 

\begin{proposition}\label{prop:properties V_L}
\begin{itemize}
\item [(i)] There exists $\varphi_0>0$ such that $V_L(x) \le x+ \varphi_0$ for all $x>0$ and $L>0$.
\item [(ii)] For any $L>0$, $V_L\in C^2((0,\infty))$ is a concave classical solution to 
\begin{equation}\label{HJB V_L}
\beta v(x)  =  \frac{1}{2} \sigma ^2 x^2v''(x) + (x^\alpha -\mu x)v'(x) + \tilde U_L (x, v'( x))\quad \hbox{for}\ x\in (0,\infty),     
\end{equation}
where $\tilde U_L:(0,\infty)^2\to (0,\infty)$ is defined by 
\begin{align*}
\tilde U_L(x, p) := \sup_{0\leq c\leq L} \{U (cx) - cxp\}.
\end{align*}
\end{itemize}
\end{proposition}

The proof of Proposition~\ref{prop:properties V_L} is relegated to Appendix~\ref{sec:appendix}, where arguments in \cite{MZ08} are extended to infinite horizon. While this extension can mostly be done in a straightforward way, there are technicalities that require detailed, nontrivial analysis. This includes, particularly, the derivation of the dynamic programming principle for $V_L$; see  Lemma~\ref{lem:viscosity V_L} for details.  

Given that $\{V_L\}_{L>0}$ is by definition a nondecreasing sequence of functions, we define 
\begin{equation}\label{V_infty}
V_\infty(x) := \lim_{L \to \infty} V_L(x)\quad \hbox{for}\ x>0.  
\end{equation}

\begin{remark}\label{rem:properties V_infty}
$V_\infty$ immediately inherits many properties from $V_L$'s.
\begin{itemize}
\item [(i)] Thanks to Proposition~\ref{prop:properties V_L}, $V_\infty$ is concave, nondecreasing, and satisfies 
\begin{equation}\label{linear bdd}
0\le V_\infty(x)\le x+\varphi_0\quad \forall x> 0.
\end{equation}
\item [(ii)] The concavity of $V_\infty$ implies that it is continuous on $(0,\infty)$. Hence, by Dini's theorem, $V_L$ converges uniformly to $V_\infty$ on any compact subset of $(0,\infty)$.
\end{itemize}
\end{remark}

\begin{lemma}\label{lem:viscosity V_infty}
$V_\infty$ is a continuous viscosity solution to
\begin{equation}\label{HJB V_infty}
\beta v(x)  =  \frac{1}{2} \sigma ^2 x^2v''(x) + (x^\alpha -\mu x)v'(x) + \tilde U (v'( x))\quad \hbox{for}\ x\in (0,\infty),     
\end{equation}
where $\tilde U:(0,\infty)\to (0,\infty)$ is defined by 
\begin{align*}
\tilde U(p) := \sup_{y\ge 0} \{U (y) - y p\}.
\end{align*}
\end{lemma}

\begin{proof}
By \eqref{U1} and \eqref{U2}, for any $p>0$, there exists a unique maximizer $y^*(p)>0$ such that $\tilde U(p) = U(y^*(p))-y^*(p)p$, and the map $p\mapsto y^*(p)$ is continuous. It follows that $\tilde U_L(x,p) = U(c^*(x,p) x)-c^*(x,p)xp$, where $c^*(x,p):= \min\{y^*(p)/x,L\}$. From these forms of $\tilde U$ and $\tilde U_L$, we see that $\tilde U_L$ converges uniformly to $\tilde U$ on any compact subset of $(0,\infty)^2$. This, together with Remark~\ref{rem:properties V_infty} (ii), implies that we can invoke the stability result of viscosity solutions (see e.g. \cite[Theorem 4.5.1]{Morimoto-book-10}). We then conclude from the stability and Proposition~\ref{prop:properties V_L} (ii) that $V_\infty$ is a viscosity solution to \eqref{HJB V_infty}.   
\end{proof}

In fact, the convergence of $V_L$ to $V_\infty$ is highly desirable. As the next result demonstrates, not only $V_L$ but also $V'_L$ and $V''_L$ converge uniformly. This readily implies smoothness of the limiting function $V_\infty$. 

\begin{proposition}\label{prop:classical V_infty}
$V'_{L}$ and $V_L''$ converge uniformly, up to a subsequence, on any compact subset of $(0,\infty)$. Hence, $V_\infty$ belongs to $C^2((0,\infty))$ with $V'_\infty(x) = \lim_{L\to\infty} V'_L(x)$ and $V''_\infty(x) = \lim_{L\to\infty} V''_L(x)$, up to a subsequence, for each $x>0$. Furthermore, $V_\infty$ is a classical solution to \eqref{HJB V_infty}.  
\end{proposition}

\begin{proof}    
Fix a compact subset $E$ of $(0,\infty)$. Let $a:= \inf E >0$ and $b:= \sup E$. 
For any $L>0$, since $V_{L}$ is nonnegative, nondecreasing, concave, and bounded above by $x + \varphi_0$ (Proposition~\ref{prop:properties V_L}), 
\[ 
0\le V_{L}'(x) \leq \frac{V_{L} (x) - V_{L}(0^+)}{x} \leq  \frac{x + \varphi_0}{x} = 1 + \frac{\varphi_0}{x}\le 1+\frac{\varphi_0}{a},\quad \forall x\in E.
\] 
Thus, $\{V_{L}'(x)\}_{L>0}$ is uniformly bounded on $E$. 

Next, we claim that $\big\{\tilde U_{L} (x, V_{L}'(x))\big\}_{L>0}$ is also uniformly bounded on $E$. To this end, we will show that there exists $C_E>0$ such that $V_{L}'(b) \geq C_E$ for all $L>0$. Assume to the contrary that there exits a subsequence $\{L_n\}_{n\in\N}$ such that $V_{L_n}'(b)\downarrow 0$. For any $x>b$, by the concavity of $V_{L_n}$, we have $V'_{L_n}(u) \leq V'_{L_n}(b)$ for $u\in[b,x]$, for all $n\in\N$. Taking integrals on both sides from $b$ to $x$ yields 
\[
V_{L_n}(x) -V_{L_n}(b) \leq V'_{L_n}(b)(x-b)\quad  \forall n\in\N. 
\]
As $n\to\infty$, we obtain $V_\infty(x) \leq V_\infty(b)$. Since $V_\infty$ is nondecreasing (Remark~\ref{rem:properties V_infty} (i)), we conclude that $V_\infty(x) = V_\infty(b)$ for all $x>b$, which in particular implies $V'_\infty(x) = V''_\infty(x) =0$ for all $x>b$. By the viscosity solution property of $V_\infty$ (Lemma~\ref{lem:viscosity V_infty}), for any $x>b$ we have $\beta V_\infty(x)=\tilde U (0)=\infty$, a contradiction. Now, with $V_{L}'(b) \geq C_E$ for all $L>0$, we have
\[
0 \leq \tilde U_{L}(x, V'_{L}(x) ) \leq \tilde U_{L}(x, V'_{L}(b)) \leq \tilde U_{L}(x, C_E) \leq \tilde U(C_E) <\infty, \quad \forall x\in E\ \hbox{and}\ L>0,
\]
where the second and the third inequalities follow from $V'_L(x)\ge V'_L(b)\ge C_E$ and $p\mapsto \tilde U_L(x,p)$ is by definition nonincreasing. This shows that $\big\{\tilde U_{L} (x, V_{L}'(x))\big\}_{L>0}$ is uniformly bounded on $E$.

Recall from Proposition~\ref{prop:properties V_L} that each $V_{L}$ satisfies  
\begin{equation}\label{HJB V_L'}
\beta V_{L}(x)  =  \frac{1}{2} \sigma ^2 x^2V''_{L}( x) + (x^\alpha -\mu x)V'_{L}(x) + \tilde U_{L} (x, V'_{L}( x)),\quad \forall x>0.
\end{equation}
By the uniform boundedness on $E$ of $\{(x^\alpha -\mu x)V'_{L}(x)\}_{L>0}$, $\{\tilde U_{L} (x, V'_{L}(x))\}_{L>0}$, and $\{V_L(x)\}_{L>0}$ (thanks to Proposition~\ref{prop:properties V_L}), \eqref{HJB V_L'} entails the uniform boundedness of  $\{V''_{L}(x)\}_{L>0}$ on $E$. By the Arzela Ascoli Theorem, this implies $V'_L$ converges uniformly, up to some subsequence, on $E$. 
With $V_L$, $V_L'$, and $\tilde U_L$ all converging uniformly on $E$ (recall from the proof of Lemma~\ref{lem:viscosity V_infty} that $\tilde U_L$ converges uniformly to $\tilde U$), \eqref{HJB V_L'} implies that  $V''_L$ also converges uniformly on $E$.   

Now, with $V_L$ converging to $V_\infty$ and $V'_L$ converging uniformly on $E$, $V_\infty$ must be continuously differentiable with $V'_\infty = \lim_{L\to\infty} V'_L$ (up to some subsequence) in the interior of $E$. This, together with $V''_L$ converging uniformly on $E$, shows that $V'_\infty$ is continuously differentiable with $V''_\infty = \lim_{L\to\infty} V''_L$ (up to some subsequence) in the interior of $E$. Since $E$ is arbitrarily chosen, we conclude that $V_\infty\in C^2((0,\infty))$. In view of Lemma~\ref{lem:viscosity V_infty}, $V_\infty$ is a classical solution to \eqref{HJB V_infty}. 
\end{proof}

\begin{remark}\label{rem:strictly increasing V_infty}
In deriving the uniform boundedness of $\{\tilde U^{L} (x, V_{L}'(x))\}_{L>0}$ in the proof above, we particularly show that $V_\infty$ is strictly increasing on $(0,\infty)$, otherwise the viscosity solution property of $V_\infty$ (Lemma~\ref{lem:viscosity V_infty}) would be violated.
\end{remark}

Now, a verification argument connects $V_\infty$ to our value function $V$. 

\begin{theorem}\label{thm:V=V_infty}
The value function $V$ in \eqref{V} coincides with $V_\infty$ on $(0,\infty)$. Hence, $V$ is concave, strictly increasing, satisfies \eqref{linear bdd}, and solves \eqref{HJB V_infty} in the classical sense.   
\end{theorem}

\begin{proof}
Since $V_\infty$ is nonnegative, concave, and nondecreasing (Remark~\ref{rem:properties V_infty} (i)), $0\le V'_\infty(x) \leq V_\infty(x)/x$ for all $x>0$. Fix $x>0$. Then, for any $T>0$ and $c\in\C$, 
\begin{align*}
\E\left[\int_{0}^{T} ( e^{-\beta s}V'_\infty( X_s) X_s)^2 ds\right] &\leq  \E\left[\int_{0}^{T} ( e^{-\beta s}V_\infty( X_s))^2 ds\right]  \leq  \E\left[\int_{0}^{T} ( e^{-\beta s} ( X_s+\varphi_0))^2 ds\right] < \infty,
\end{align*}
where the second line follows from Remark~\ref{rem:properties V_infty} (i) and the finiteness is due to \eqref{eq2.5}. It follows that $\int_{0}^{t} e^{-\beta s}V_x( X_s) X_s dW_s$ is a martingale on $[0,T]$, for any $T>0$ and $c\in\C$. Now, fix $c\in\C$. By using Ito's formula, for any $T>0$, 
\begin{align}
\E[e^{-\beta T} V_\infty(X_T)] &= V_\infty(x)\notag\\
&+ \E\bigg[\int_{0}^{T} e^{-\beta t} \bigg(-\beta V_\infty(X_t) +  V'_\infty(X_t) ( X_t^{\alpha}  - \mu X_t- c_t X_t)dt + \frac{\sigma^2}{2}X_t^{2} V''_{\infty}(X_t) dt \bigg)\bigg] \notag\\ 
&\le V_\infty(x) - \E\left[\int_{0}^{T} e^{-\beta t} U(c_tX_t)dt  \right],\label{lalala}
\end{align}
where the inequality follows from $V_\infty$ satisfying \eqref{HJB V_infty} (Proposition~\ref{prop:classical V_infty}). As $T\to\infty$, we deduce from Remark~\ref{rem:properties V_infty} (i) and \eqref{eq2.5} that
\[
\E[e^{-\beta T} V_\infty(X_T) ] \leq \E\left[e^{-\beta T} (X_T+\varphi_0) \right] \leq e^{-\beta T} (2^{\eta -1} (x + T^{\eta}) +\varphi_0) \rightarrow 0,\quad \hbox{as}\ T\to\infty.  
\] 
Thus, we conclude from \eqref{lalala} that $V_\infty(x) \ge \E\big[\int_{0}^{\infty} e^{-\beta t} U(c_t X_t) dt \big]$ for all $c\in\C$, and thus $V_\infty(x) \ge V(x)$. On the other hand, by definition $V(x)\ge V_L(x)$ for all $L>0$, and thus $V(x)\ge V_\infty(x)$. We therefore conclude that $V(x)=V_\infty(x)$. 
The remaining assertions follow from Remark~\ref{rem:properties V_infty} (i), Remark~\ref{rem:strictly increasing V_infty}, and Proposition~\ref{prop:classical V_infty}. 
\end{proof}

While Theorem~\ref{thm:V=V_infty} associates $V$ with the nonlinear elliptic equation \eqref{HJB V_infty}, this is {\it not} a full characterization of $V$, as there may be multiple solutions to \eqref{HJB V_infty}. To further characterize $V$ as the {\it unique} classical solution to \eqref{HJB V_infty} among a certain class of functions, the standard approach is to stipulate an optimal control of feedback form, by which one can complete the verification argument; note that the proof of Theorem~\ref{thm:V=V_infty} amounts to the first half of the verification argument. 

As detailed in Section~\ref{sec:optimal consumption} below, although the form of a candidate optimal consumption process $\hat c$ can be readily read out from the equation \eqref{HJB V_infty}, it is highly nontrivial whether $\hat c$ is a well-defined stochastic process, due to the unboundedness of $\hat c$. This entails additional analysis of the value function $V$ and the capital per capita process $X$, as we will now introduce.


\section{Optimal Consumption}\label{sec:optimal consumption}

In view of \eqref{HJB V_infty}, one can heuristically stipulate the form of an optimal consumption process as
\begin{equation}\label{hat c}
\hat c_t := \hat c(X_t)\quad \hbox{for}\ t\ge 0,\qquad \hbox{with}\quad \hat c(x) := \frac{(U')^{-1}(V'(x))}{x}\quad \hbox{for}\ x>0, 
\end{equation}
where $X$ is the solution to the SDE \eqref{X} with $c_t$ replaced by $\hat c_t$, i.e. the solution to
\begin{equation}\label{X'}
dX_t = \left(X_t^\alpha  - \mu X_t - (U')^{-1}\left(V'(X_t)\right)  \right)dt - \sigma X_t dW_t,    \quad X_0 = x>0.
\end{equation}
For $\hat c$ in \eqref{hat c} to be well-defined, two questions naturally arise. First, it is unclear whether \eqref{X'} admits a solution: Proposition~\ref{prop:existence} is an existence result for \eqref{X}, specifically when $c$ is an {\it a priori} given process, without $X_t$ involved. 
Second, even if a solution $X$ to \eqref{X'} exists, it is in question whether $X$ is strictly positive, so that one does not need to worry about the problematic case ``$X_t =0$'' in \eqref{hat c}.  

For \eqref{X'} to admit a solution, we first observe that it is necessary to have $V'(0+) = \infty$. Indeed, if $c:= V'(0+) < \infty$, when $X$ is close enough to zero, the drift coefficient of \eqref{X'} will approach the constant $- (U')^{-1}\left(c\right)<0$, while the diffusion coefficient will tend to zero. This will eventually bring $X$ down to zero. When this happens, the drift and the diffusion coefficients will be precisely $- (U')^{-1}\left(c\right)<0$ and $0$ respectively, which will move $X$ further to take negative values. The drift coefficient of \eqref{X'}, however, is not well-defined for negative values of $X_t$. A solution to \eqref{X'}, as a result, cannot exist if  $V'(0+) < \infty$. 

The next result analyzes the behavior of $V$ as $x\downarrow 0$, and particularly establishes $V'(0+) = \infty$. 

\begin{lemma}\label{lem:V at 0}
The function $V$ defined in \eqref{V} satisfies the following:
\begin{itemize}
\item [(i)]  $V(0+) > 0$. 
\item [(ii)] Assume $U\in C^2((0,\infty))$. As $x\downarrow 0$, $V'$ explodes and is of the order of $x^{-\alpha}$. Specifically, 
\[
V'(0+) = \infty\quad \hbox{and}\quad \lim_{x\to 0+}x^\alpha V'(x) = \beta V(0+)>0.
\] 
Furthermore,
\begin{equation}\label{C*/x^alpha}
\lim_{x \downarrow 0} \frac{(U')^{-1} (V'(x))}{x^\alpha} = 0.
\end{equation}
\end{itemize}
\end{lemma}

\begin{proof}
(i) Consider $\bar c\in\C$ with $\bar c \equiv 1$. For any $x>0$, in view of \eqref{Z formula}, the corresponding capital per capita process $X^x_t$ is given by 
\[
X^x_t = G_t^{-\frac{1}{1-\alpha}}\left(x^{1-\alpha}+(1-\alpha)\int_0^t G_s ds\right)^{\frac{1}{1-\alpha}}, 
\]
where $G_t$ is given as in \eqref{G formula} with $c_t$ replaced by the constant $1$. Then, by the definition of $V$, 
\[
V(x) \ge \E\left[\int_0^\infty e^{-\beta t} U(X^x_t)dt\right] = \E\left[\int_0^\infty e^{-\beta t} U\left( G_t^{-\frac{1}{1-\alpha}}\left(x^{1-\alpha}+(1-\alpha)\int_0^t G_s ds\right)^{\frac{1}{1-\alpha}}\right) dt\right].
\] 
As $x\downarrow 0$, Fatou's lemma gives $V(0+)\ge \E\big[\int_0^\infty e^{-\beta t} U(((1-\alpha)\int_0^t G_{s,t}^{-1} ds)^{\frac{1}{1-\alpha}} dt\big]>0$, where $G_{s,t}$ is given as in \eqref{G_s,t} with $c_t$ replaced by the constant $1$.  

(ii) By contradiction, assume that $c := V'(0+) < \infty$. Note that $c>0$ must hold, as $V$ is concave and strictly increasing (Theorem~\ref{thm:V=V_infty}). Consider $I(y) := (U')^{-1}(y)$ for $y\in(0,\infty)$. With $U\in C^2((0,\infty))$, the inverse function theorem implies that $I\in C^1((0,\infty))$ with $I'(y)=1/U''(y)$. Thanks again to Theorem~\ref{thm:V=V_infty}, we have
\begin{equation}\label{HJB V_infty'}
\beta V(x)  =  \frac{1}{2} \sigma ^2 x^2V''(x) + (x^\alpha-\mu x) V'(x) + U(I(V'(x)))- I(V'(x))V'(x),\quad \forall x>0.
\end{equation}
We can then express $V''(x)$ in terms of the functions $x$, $V(x)$, $V'(x)$, $I(V'(x))$, and $U(I(V'(x)))$. Since each of these functions is continuously differentiable, we have $V\in C^3((0,\infty))$. By using L'Hospital's rule, 
\begin{equation}\label{L'Hospital'}
c = \lim_{x \downarrow 0}V'(x) = \lim_{x \downarrow 0}\frac{xV'(x)}{x} = \lim_{x \downarrow 0} \left(V'(x) + xV''(x)\right), 
\end{equation}
which implies $\lim_{x \downarrow 0}xV''(x) =0$. The same argument in turn gives 
\begin{align*}
0 = \lim_{x \downarrow 0}xV''(x) = \lim_{x \downarrow 0}\frac{x^2V''(x)}{x} = \lim_{x \downarrow 0} \left(2xV''(x) + x^2V'''(x) \right),
\end{align*}
leading to $\lim_{x \downarrow 0}x^2V'''(x)=0$. Now, by differentiating both sides of \eqref{HJB V_infty'} and multiplying them by $x^{1-\alpha}$, we get 
\begin{align}
\beta x^{1-\alpha} V'(x) &=  \sigma^2  x^{2-\alpha} V''(x)+ \frac{1}{2} \sigma^2 x^{3-\alpha} V'''(x) +x V''(x) + \alpha V'(x)\notag\\
&\ \ \  -\mu x^{1-\alpha} V'(x) - \mu x^{2-\alpha} V''(x)  -x^{1-\alpha} I(V'(x)) V''(x),\label{HJB'''}    
\end{align}
where the last term is obtained by noting that $U'\circ I$ is the identity map. As $x\downarrow 0$ in \eqref{HJB'''}, we get
\begin{align*}
0 = \alpha c  + \lim_{x \downarrow 0} x^{1-\alpha} I(V'(x)) (-V''(x)).
\end{align*}
This is a contradiction by noting that $\alpha c>0$ and the limit above is nonnegative (as $I$ is a positive function and $V$ is concave). We therefore conclude that $V'(0+) = \infty$.

Now, since $V$ satisfies \eqref{linear bdd} (Theorem~\ref{thm:V=V_infty}), we have $\limsup_{x\downarrow 0}x V'(x)<\infty$. Take an arbitrary sequence $\{x_n\}_{n\in\N}$ such that $x_n\downarrow 0$ and $x_n V'(x_n)$ converges as $n\to\infty$. Let $\ell:= \lim_{n\to\infty}x_n V'(x_n) < \infty$. Similarly to \eqref{L'Hospital'},  
\begin{equation*}
\ell = \lim_{n \to \infty} x_n V'(x_n) = \lim_{n \to \infty} \frac{x_n^2 V'(x_n)}{x_n} =\lim_{n \to \infty} \left(2x_n V'(x_n) +  x_n^2 V'' (x_n)\right) = 2\ell + \lim_{n \to \infty}x_n^2 V'' (x_n),
\end{equation*}
which yields $\lim_{n \to \infty}x_n^2 V''(x_n) = -\ell$. Recalling that $V$ is a classical solution to \eqref{HJB V_infty}, we have 
\[
\beta V(x_n)  =  \frac{1}{2} \sigma ^2 x_n^2V''(x_n) + (x_n^\alpha -\mu x_n)V'(x_n) + \tilde U (V'(x_n))\quad \hbox{for all}\ n\in\N.     
\]
As $n\to\infty$, since $V'(0+)=\infty$ implies $\tilde U (V'(x_n))\to 0$, we obtain 
\[
\beta V(0+) = -\left(\frac{1}{2}\sigma^2 +\mu\right) \ell + \lim_{n\to\infty} x_n^\alpha V'(x_n).
\]
If $\ell>0$, then $\lim_{n\to\infty} x_n^\alpha V'(x_n) = \ell \lim_{n\to\infty}x_n^{\alpha-1} = \infty$, which would violate the above equality. Thus, $\ell=0$ must hold. Since $\{x_n\}_{n\in\N}$ above is arbitrarily chosen, we conclude that $\lim_{x\downarrow 0} x^\alpha V'(x) = \beta V(0+)>0$, where the inequality follows from (i).

Finally, to prove \eqref{C*/x^alpha}, observe that
$0 \leq \tilde U(V'(x) ) 
= U ((U')^{-1}(V'(x))) - V'(x)(U')^{-1}(V'(x))$ for all $x>0$,  leading to 
\[  
0 \leq V'(x)(U')^{-1}(V'(x)) \leq U ((U')^{-1}(V'(x)))\quad \forall x>0.
\] 
As $x\downarrow 0$, since $V'(0+) = \infty$ and $U(0) = 0$, the right hand side above approaches zero, which implies $$\lim_{x \downarrow 0} V'(x)(U')^{-1}(V'(x))=0.$$ This, together with $\lim_{x \downarrow 0} x^\alpha V'(x) = \beta V(0+)>0$, gives \eqref{C*/x^alpha}. 
\end{proof}

On the strength of Lemma~\ref{lem:V at 0}, we are ready to present the existence result for \eqref{X'}.

\begin{proposition}\label{prop:X'}
Suppose $U\in C^2((0,\infty))$. For any $x>0$, there exists a unique strong solution to \eqref{X'}, which is strictly positive a.s.
\end{proposition}

\begin{proof}
We will first establish the existence of a weak solution to \eqref{X'}, which is strictly positive a.s. Then, we will prove that pathwise uniqueness holds for \eqref{X'}. By \cite[Section 5.3.D]{KS-book-91}, this gives the desired result that a unique strong solution exists and it is strictly positive a.s.

{\bf Step 1: Construct a weak solution to \eqref{X'} that is strictly positive a.s.} Thanks to the argument in \cite[Theorem 5.5.15]{KS-book-91}, with $\R$ replaced by $(0,\infty)$, there exists a weak solution $X$ to \eqref{X'} up to the explosion time
\[
S := \lim_{n\to\infty} S_n,\quad \hbox{where}\quad S_n := \inf\{t\ge 0 : X_t\notin (1/n,n)\}.
\] 
We will show that $\P(S=\infty)=1$. In view of Feller's test for explosion (see e.g. \cite[Theorem 5.5.29]{KS-book-91}), as well as \cite[Theorem 5.5.27]{KS-book-91}, it suffices to prove that for any $\ell\in (0,\infty)$,
\begin{equation}\label{A1}
A_1:= \int_{\ell}^\infty \exp\left(-2\int_r^\ell \frac{y^\alpha-\mu y -(U')^{-1}(V'(y))}{\sigma^2 y^2} dy \right) dr = \infty,
\end{equation}
and
\begin{equation}\label{A2}
A_2:= \int_{0+}^\ell \exp\left(2\int_r^\ell \frac{y^\alpha-\mu y -(U')^{-1}(V'(y))}{\sigma^2 y^2} dy \right) dr = \infty.
\end{equation}
Let $C_1:= \exp\left(-\frac{2}{\sigma^2} \left(\frac{\ell^{\alpha-1}}{1-\alpha} + \mu \log(\ell)\right)\right)>0$. Observe that 
\begin{align*}
A_1 &\ge \int_{\ell}^\infty \exp\left(-2\int_r^\ell \frac{y^\alpha-\mu y}{\sigma^2 y^2} dy \right) dr = C_1 \int_{\ell}^\infty \exp\left(\frac{2}{\sigma^2 (1-\alpha)} \left(\frac{1}{r}\right)^{1-\alpha} \right) r^{\frac{2\mu}{\sigma^2}} dr\\
&\ge C_1 \int_\ell^\infty  r^{\frac{2\mu}{\sigma^2}} dr =\infty,
\end{align*}
which gives \eqref{A1}. On the other hand, by \eqref{C*/x^alpha}, there exists $0<\delta< \ell$ such that $(U')^{-1}(V'(y))<\frac{1}{2} y^\alpha$ for $0<y<\delta$. It follows that
\begin{align*}
A_2 &\ge \int_{0+}^\delta \exp\left(2\int_r^\ell \frac{y^\alpha-\mu y -(U')^{-1}(V'(y))}{\sigma^2 y^2} dy \right) dr \\
&= \int_{0+}^\delta \exp\left(2\int_r^\delta \frac{y^\alpha-\mu y -(U')^{-1}(V'(y))}{\sigma^2 y^2} dy + 2\int_\delta^\ell \frac{y^\alpha-\mu y -(U')^{-1}(V'(y))}{\sigma^2 y^2} dy \right) dr\\
&= C_2 \int_{0+}^\delta \exp\left(2\int_r^\delta \frac{y^\alpha-\mu y -(U')^{-1}(V'(y))}{\sigma^2 y^2} dy \right) dr\\
&\ge C_2 \int_{0+}^\delta \exp\left(\frac{2}{\sigma^2}\int_r^\delta \frac{1}{2} y^{\alpha-2} - \mu y^{-1} dy  \right) dr\\
&\ge C_2 C_3 \int_{0+}^\delta \exp\left(\frac{1}{\sigma^2(1-\alpha)} \left(\frac{1}{r}\right)^{1-\alpha}\right) r^{\frac{2\mu}{\sigma^2}} dr = \infty,
\end{align*}
where $C_2 := \exp\left(2\int_\delta^\ell \frac{y^\alpha-\mu y -(U')^{-1}(V'(y))}{\sigma^2 y^2} dy\right)$, $C_3 := \exp\left(\frac{-\delta^{\alpha-1}}{\sigma^2(1-\alpha)}\right) \delta^{-\frac{2\mu}{\sigma^2}}$, and the fourth line above follows from $(U')^{-1}(V'(y))<\frac{1}{2} y^\alpha$ for $0<y<\delta$. This readily shows \eqref{A2}. We therefore conclude that the weak solution $X$ takes values in $(0,\infty)$ a.s. 

{\bf Step 2: Show that pathwise uniqueness holds for \eqref{X'}.} Let $x^* >0$ be the unique maximizer of $\sup_{x\ge 0}\{x^\alpha - \mu x\}$. Observe that $x\mapsto x^\alpha-\mu x$ is strictly increasing on $(0,x^*)$ and strictly decreasing on $(x^*,\infty)$. Also, the concavity of $V$ (Theorem~\ref{thm:V=V_infty}) implies that $V'$ is nonincreasing. Since $U$ is strictly concave, $U'$ is strictly decreasing, and so is $(U')^{-1}$. It follows that $x\mapsto (U')^{-1}(V'(x))$ is nondecreasing. 
We then conclude that the drift coefficient $b(x) := x^\alpha-\mu x - (U')^{-1}(V'(x))$ of \eqref{X'} is strictly decreasing on $(x^*,\infty)$. 

Besides the weak solution $X$ in Step 1, let $\overline X$ be another weak solution to \eqref{X'}, with $(\Omega, \F,\P)$, $W$, and the initial value $x>0$ all the same as those of $X$. By the same argument in Step 1, $\overline X$ takes values in $(0,\infty)$ a.s. For each $N\in\N$, consider
\[
\tau_N := \inf\{t\ge 0 : X_t \le 1/N 
\}.
\]  
We claim that for any $x>0$,
\begin{equation}\label{pu}
\P\left(X^x_{t\wedge \tau_N} = \overline X^x_{t\wedge \tau_N},\ \forall t\ge 0\right) =1,\quad \forall N\in\N.
\end{equation} 

Pick an arbitrary $\eps>0$, and let $x_0:= x^*+\eps$. Fix $N\in\N$. If the initial value $x< x_0$, since the diffusion coefficient $a(u):=\sigma u$ of \eqref{X'} is bounded away from zero on $[1/N,x_0]$, the argument in \cite[Theorem]{Nakao72} (with $c$ and $M$ therein replaced by $\sigma/N$ and $\sigma x_0$ in our case) implies 
\begin{equation}\label{pu1}
\P\left(X^x_{t\wedge \tau_N\wedge \tau_{x_0}} = \overline X^x_{t\wedge \tau_N\wedge \tau_{x_0}},\ \forall t\ge 0\right) =1,
\end{equation}
where $\tau_0 := \inf\{t\ge 0 : X^x_t \ge x_0\}$. On the other hand, if the initial value $x\ge x_0$, since the drift coefficient $b(u)$ of \eqref{X'} is strictly decreasing on $(x^*,\infty)$, \cite[Example 1.1]{Yamada73} asserts that 
\begin{equation}\label{pu2}
\P\left(X^x_{t\wedge \tau_{x^*}} = \overline X^x_{t\wedge  \tau_{x^*}},\ \forall t\ge 0\right) =1,
\end{equation}
where $\tau_{x^*} := \inf\{t\ge 0 : X^x_t \le x^*\}$. Note that \eqref{pu1} and \eqref{pu2} already imply the desired result \eqref{pu}. Indeed, if the initial value $x<x_0$, we can define a sequence of stopping times recursively as follows: $\tau_0 := 0$, 
\[
\tau_{2n-1} : = \inf\{t\ge \tau_{2n-2} :  X^x_t \ge x_0\},\quad \tau_{2n} := \inf\{t\ge \tau_{2n-1}: X^x_t \le x^*\},\quad \forall n\in\N. 
\]
Then, by using \eqref{pu1} and \eqref{pu2} alternately on the time intervals $[\tau_{n-1},\tau_n]$, $n=1,2,...$, we obtain \eqref{pu}. If the initial value $x\ge x_0$, we can similarly define a sequence of stopping times recursively as follows: $\tau_0 := 0$, 
\[
\tau_{2n-1} : = \inf\{t\ge \tau_{2n-2} : X^x_t \le x^* \},\quad \tau_{2n} := \inf\{t\ge \tau_{2n-1}:  X^x_t \ge x_0\},\quad \forall n\in\N. 
\]
By applying \eqref{pu2} and \eqref{pu1} alternately on the time intervals $[\tau_{n-1},\tau_n]$, $n=1,2,...$, we again obtain \eqref{pu}. 

Finally, since $X$ is strictly positive a.s., $\tau_N\to\infty$ a.s. as $N\to\infty$. We then conclude from \eqref{pu} that $\P\left(X^x_{t} = \overline X^x_{t},\ \forall t\ge 0\right) =1$, for all $x>0$. That is, pathwise uniqueness holds for \eqref{X'}, as desired. 
\end{proof}

\begin{remark}\label{rem:Girsanov}
With bounded consumptions and a finite horizon $T>0$, \cite[Lemma 6.1]{MZ08} constructs a strictly positive solution to \eqref{X'} easily, through a change of measure and using Girsanov's theorem. This does not work in our case. With unbounded consumptions, the same change of measure is not well-defined. Also, applying Girsanov's theorem requires some finite horizon. In view of this, Proposition~\ref{prop:X'} complements \cite[Lemma 6.1]{MZ08}, by providing a new, different construction that accommodates both unbounded consumptions and infinite horizon. 
\end{remark}

Proposition~\ref{prop:X'} deals with the SDE \eqref{X'}, induced by the value function $V$. In fact, the same arguments can be applied to SDEs induced by a much larger class of functions. 

\begin{corollary}\label{coro:X' with u}
Suppose $U\in C^2((0,\infty))$. Let $u\in C^1((0,\infty))$ be strictly increasing, concave, and satisfy
\begin{equation}\label{C*/x^alpha'}
\lim_{x \downarrow 0} \frac{(U')^{-1} (u'(x))}{x^\alpha} = 0.
\end{equation}
Then, for any $x>0$, the SDE 
\begin{equation}\label{X''}
dX_t = \left(X_t^\alpha  - \mu X_t - (U')^{-1}\left(u'(X_t)\right)  \right)dt - \sigma X_t dW_t,    \quad X_0 = x
\end{equation}
admits a unique strong solution, which is strictly positive a.s.
\end{corollary}

\begin{proof}
The result can be established by following the proof of Proposition~\ref{X'}, with $V$ replaced by $u$. Specifically, Step 1 in the proof can be carried out thanks to $u'(x)>0$ and \eqref{C*/x^alpha'}, while Step 2 relies on the concavity of $u$.
\end{proof}

Let $\mathcal U$ denote the class of functions $u\in C^2((0,\infty))\cap C([0,\infty))$ that are nonnegative, strictly increasing, concave, satisfying \eqref{C*/x^alpha'} and the following linear growth condition: there exists $C>0$ such that 
\begin{equation}\label{linear growth}
u(x)\le C(1+x)\quad \hbox{for all}\ x\ge 0. 
\end{equation}

Now, we are ready to present the main result of this paper.

\begin{theorem}\label{thm:main}
Suppose $U\in C^2((0,\infty))$. The function $V$ defined in \eqref{V} is the unique classical solution to \eqref{HJB V_infty} among functions in $\mathcal U$. Moreover, $\hat c\in \C$ defined by \eqref{hat c}, with $X$ being the unique strong solution to \eqref{X'}, is an optimal consumption process for \eqref{V}.  
\end{theorem}

\begin{proof}
We know from Theorem~\ref{thm:V=V_infty} and Lemma~\ref{lem:V at 0} that $V\in \mathcal U$ and it solves \eqref{HJB V_infty} in the classical sense. By following the arguments in Theorem~\ref{thm:V=V_infty}, with $V_\infty$ and $c$ therein replaced by $V$ and $\hat c$, we note that the inequality in \eqref{lalala} now becomes equality, leading to $V(x) =  \E\big[\int_{0}^{\infty} e^{-\beta t} U(\hat c_t X^x_t) dt \big]$ for all $x>0$. This readily shows that $\hat c\in \C$ is an optimal consumption process for \eqref{V}.

For any $u\in \mathcal U$ that solves \eqref{HJB V_infty} in the classical sense, we can again follow the arguments in Theorem~\ref{thm:V=V_infty} to show that $u\ge V$. On the other hand, consider the consumption process
\begin{equation}\label{hat c^u}
\hat c^u_t :=  \hat c^u(X_t)\quad \hbox{for}\ t\ge 0,\qquad \hbox{with}\quad \hat c^u(x) := \frac{(U')^{-1}(u'(x))}{x}\quad \hbox{for}\ x>0, 
\end{equation}
where $X$ is the unique strong solution to \eqref{X''}, whose existence is guaranteed by Corollary~\ref{coro:X' with u}. Now, in \eqref{lalala}, if we replace $V_\infty$ and $c$ therein by $u$ and $\hat c^u$, the inequality becomes equality, leading to $u(x) =  \E\big[\int_{0}^{\infty} e^{-\beta t} U(\hat c^u_t X^x_t) dt \big]\le V(x)$ for all $x>0$. Thus, we conclude that $u=V$. 
\end{proof}

\begin{remark}\label{rem:characterize V_L}
In the characterization of $V$ in Theorem~\ref{thm:main}, condition \eqref{C*/x^alpha'} is the key to dealing with unbounded consumptions (recall that \eqref{C*/x^alpha'} is part of the definition of  $\mathcal U$). If we restrict ourselves to $\mathcal C_L$ in \eqref{C_L} for some $L>0$ (as in \cite{MZ08}), there is no need to impose \eqref{C*/x^alpha'}.

To see this, note that \eqref{C*/x^alpha'} can be re-written as
\[
\lim_{x\downarrow 0} \hat c^u(x) x^{1-\alpha} = 0,\quad \hbox{with $\hat c^u$ as in \eqref{hat c^u}}.
\]
That is, we require the optimal consumption to be dominated by $x^{1-\alpha}$ as $x\downarrow 0$. When we are restricted to $\mathcal C_L$, this requirement holds trivially, thanks to the bound $L>0$ for each $c\in \mathcal C_L$. Thus, for $V_L$ defined in \eqref{V_L}, the same arguments in Proposition~\ref{prop:X'}, Corollary~\ref{coro:X' with u}, and Theorem \ref{thm:main} can be carried out, without the need to impose \eqref{C*/x^alpha'}. This leads to the characterization: $V_L$ is the unique classical solution to \eqref{HJB V_L} among the class of functions $u\in C^2((0,\infty))\cap C([0,\infty))$ that are nonnegative, strictly increasing, concave, and satisfying \eqref{linear growth}.   
\end{remark}

\begin{remark}\label{rem:fills void}
In \cite{MZ08}, one is restricted to $\mathcal C_L$ in \eqref{C_L}. The main results, \cite[Theorems 4.2 and 6.2]{MZ08}, only show that the value function $V_L$ is a classical solution and that a feedback optimal consumption exists; there is no further characterization of $V_L$. At the end of \cite{MZ08}, the authors very briefly mention, without a proof, that $V_L$ is the unique solution. However, the class of functions among which $V_L$ is unique, the key ingredient of any PDE characterization, is missing. Theorem~\ref{thm:main}, along with the resulting characterization of $V_L$ in Remark~\ref{rem:characterize V_L}, fills this void. 
\end{remark}

We will demonstrate the use of Theorem~\ref{thm:main} explicitly in Proposition~\ref{prop:explicit case} below.


\subsection{Comparison with \cite{Morimoto08, Morimoto-book-10}}\label{subsec:comparison}
To the best of our knowledge, Morimoto \cite{Morimoto08, Morimoto-book-10} are the only prior works that consider unbounded consumptions in the stochastic Ramsey problem. Our studies complement \cite{Morimoto08, Morimoto-book-10} in two ways. 

First, \cite{Morimoto08, Morimoto-book-10} require the production function $F(k,y)$ to satisfy $F_k(0+,y)<\infty$ for all $y>0$. This provides technical conveniences: (i) The drift coefficient of the capital per capita process is Lipschitz (see e.g. (11) and (12) in \cite{Morimoto08}), such that the SDE has uniqueness of solutions even when the initial condition is 0. The value function $V$ is thus well-defined at $x=0$, with $V(0)=0$. (ii) The continuity of $V$ at $x=0$ is ensured, with $V(0+)=V(0) = 0$, which leads to a short simple proof for $V'(0+) = \infty$ (see the last two lines in the proof of \cite[Theorem 4.1]{Morimoto08}).  

Our contribution here is taking into account the classical, widely-used Cobb-Douglas production function \eqref{CD}, which violates $F_k(0+,y)<\infty$. In contrast to \cite{Morimoto08, Morimoto-book-10}, the drift coefficient of \eqref{X} is non-Lipschitz, such that \eqref{X} admits multiple solutions when the initial condition is 0 (see Corollary~\ref{coro:X=0}), leaving the value function $V$ undefined at $x=0$ (see Remark~\ref{rem:V(0)=0}). Moreover, proving $V'(0+) = \infty$ now requires much more involved analysis, as shown in Lemma~\ref{lem:V at 0}.   

Second, with unbounded consumptions considered, the framework in \cite{Morimoto08, Morimoto-book-10}, like ours, suffers the potential issue that the solution $X$ to \eqref{X'} may reach $0$ in finite time. The author of \cite{Morimoto08, Morimoto-book-10} does not analyze whether or not, or how likely, $X$ will reach $0$ in finite time, but simply restricts the Ramsey problem to the random horizon $[0,\tau_X]$, where $\tau_X$ is the first time $X$ reaches $0$. 
However, it is hard to imagine that in practice individuals would allow $X$, the capital per capita, to reach $0$, and enjoy no consumption at all afterwards (This is, nonetheless, what \cite[(36)]{Morimoto08} prescribes). 

In a reasonable economic model, an optimal consumption process should by itself prevents $X$ from reaching 0, so that there is no need to artificially introduce $\tau_X$.    
In this aspect, our paper complements \cite{Morimoto08, Morimoto-book-10}, by providing a framework in which $\tau_X=\infty$ is ensured under optimal consumption behavior.



\section{Comparison with Bounded Consumption in \cite{MZ08}}\label{sec:comparison}
For each $L>0$, one can solve the problem \eqref{V_L} by modifying the arguments in \cite{MZ08}, with an optimal consumption process given by
\begin{equation}\label{c_L}
\hat c_t^L := \hat c^L(X_t)\quad \hbox{for}\ t\ge 0,\qquad \hbox{with}\quad \hat c^L(x) :=  \min\bigg\{ \frac{(U')^{-1}(V_L'(x))}{x}, L\bigg\}\quad \hbox{for}\ x>0, 
\end{equation}
where $X$ is the unique strong solution to \eqref{X} with $c_t$ replaced by $\hat c^L_t$. 

Two questions are particularly of interest here. First, by switching from the bounded strategy $\hat c^L$, however large $L>0$ may be, to the possibly unbounded $\hat c$ in \eqref{hat c}, can we truly raise our expected utility? An affirmative answer will be provided below, which justifies economically the use of unbounded strategies. 
Second, for each $L>0$, do agents following $\hat c^L$ simply chop the no-constraint optimal strategy $\hat c$ at the bound $L>0$? In other words, does ``$\hat c^L = \hat c \wedge L$'' hold? As we will see, this fails in general, suggesting a more structural change from $\hat c^L$ to $\hat c$. 


Our first result shows that switching from $\hat c^L$ to $\hat c$ strictly increases expected utility at {\it all} levels of wealth (capital per capita) $x>0$, whenever $\hat c$ is truly unbounded.  

\begin{proposition}\label{prop:V>V_L}
Suppose $U\in C^2((0,\infty))$. Let $M:= \sup_{x>0} \hat c(x)$. 
\begin{itemize}
\item [(i)] If $M<\infty$, then for any $L\ge M$, $V_L(x)= V(x)$ for all $x>0$. 
\item [(ii)] If $M=\infty$, then for any $L>0$, $V_L(x)<V(x)$ for all $x>0$.  
\end{itemize}
\end{proposition}

\begin{proof}
(i) Since $\hat c$ in \eqref{hat c} is optimal for $V$ (Theorem~\ref{thm:main}) and bounded by $M<\infty$, the definitions of $V$ and $V_L$ in \eqref{V} and \eqref{V_L} directly imply $V_L=V$ for $L\ge M$.

(ii) Fix $L>0$. First, we claim that there exists $x^*\in (0,\infty)$ with $V(x^*)> V_L(x^*)$. Suppoe $V=V_L$ on $(0,\infty)$. With $M=\infty$, we can take $x >0$ with $\hat c(x)>L$. This implies $\tilde U(V'(x)) = U(\hat c(x)x)-\hat c(x)x V'(x) > U(L x)- Lx V'(x) = \tilde U_L(x, V'(x))$. By this and Theorem~\ref{thm:V=V_infty}, 
\begin{align*}
0 &= -\beta V(x) +\frac{1}{2} \sigma ^2 x^2V''(x) + (x^\alpha -\mu x)V'(x) + \tilde U (V'(x))\\
 &> -\beta V(x) +\frac{1}{2} \sigma ^2 x^2V''(x) + (x^\alpha -\mu x)V'(x)+ \tilde U_L(x, V'(x))\\
 &= -\beta V_L(x) +\frac{1}{2} \sigma ^2 x^2V_L''(x) + (x^\alpha -\mu x)V_L'(x)+ \tilde U_L(x, V_L'(x)),  
\end{align*}
where the last line follows from $V=V_L$ on $(0,\infty)$. This, however, contradicts Proposition~\ref{prop:properties V_L} (ii). 

With $V(x^*)> V_L(x^*)$ for some $x^*>0$, we will show that $V(x) > V_L(x)$ for {\it all} $x>0$. Recall the dynamic programming principle of $V_L$ in \eqref{DPP}. By using the same arguments in Lemma~\ref{lem:viscosity V_L}, one can derive the corresponding principle for $V$, i.e. for any $x>0$,
\begin{equation}\label{DPP for V}
V(x) \ge \sup_{c\in\C} \E\left[\int_{0}^{\tau} e^{-\beta t} U(c_t X^x_t) dt + e^{-\beta \tau }V(X^x_\tau)\right],\quad \forall \tau\in \T. 
\end{equation}
Now, for any $x>0$ with $x\neq x^*$, let $X$ denote the unique strong solution to \eqref{X}, with $c_t$ replaced by $\hat c^L_t$. Consider $\tau^* := \inf\{t\ge 0 : X^x_t =x^*\}\in \T$. Thanks to \eqref{DPP for V}, 
\begin{align*}
V(x) &\ge  \E\bigg[\int_{0}^{\tau^*} e^{-\beta t} U(\hat c^L_t X^x_t) dt + e^{-\beta\tau^*} V(X^x_{\tau^*})\bigg]\\ 
&> \E\bigg[\int_{0}^{\tau^*} e^{-\beta t} U(\hat c^L_t X^x_t) dt + e^{-\beta\tau^*} V_L(X^x_{\tau^*})\bigg]\\
&\ge \E\bigg[\int_{0}^{\tau^*} e^{-\beta t} U(\hat c^L_t X^x_t) dt + e^{-\beta\tau^*} \E\bigg[\int_{\tau^*}^{\infty} e^{-\beta (t-\tau^*)} U(\hat c^L_t X^x_t) dt\ \bigg|\ \F_{\tau^*}\bigg]\bigg]\\
&= \E\bigg[\int_{0}^{\infty} e^{-\beta t} U(\hat c^L_t X^x_t) dt \bigg] = V_L(x),
\end{align*}
where the second inequality is due to $V(X^x_{\tau^*}) = V(x^*)> V_L(x^*) = V_L(X^x_{\tau^*})$, the third inequality follows from the same calculation as in \eqref{DPP <}, and the last equality holds as $\hat c^L$ is optimal for $V_L$. Hence, we conclude that $V(x) > V_L(x)$ for all $x>0$.
\end{proof}

Proposition~\ref{prop:V>V_L} provides an answer to whether ``$\hat c^L = \hat c \wedge L$'' holds. 


\begin{corollary}\label{coro:statement not true}
Suppose $\sup_{x>0} \hat c(x)=\infty$. Given $L>0$, for any $x>0$ with $\hat c(x)<L$, and any $\delta>0$, there exists $x^*>0$ such that $|x^*-x|<\delta$ and $\hat c^L(x^*) \neq \hat c(x^*) \wedge L$. Hence, for any $L>\inf_{x>0} \hat c(x)$, there exists $x^*>0$ such that $\hat c^L(x^*) \neq \hat c(x^*) \wedge L$.
\end{corollary}

\begin{proof}
Take $L>0$ such that there exists $x>0$ with $\hat c(x)<L$. For any $\delta>0$, by the continuity of $\hat c$, there exists $0<\delta'\le \delta$ such that $\hat c(y)<L$ for all $y\in (x-\delta',x+\delta')$. We claim that there exists $y^*\in (x-\delta',x+\delta')$ such that $\hat c^L(y^*) \neq \hat c(y^*)\wedge L$. By contradiction, suppose $\hat c^L = \hat c\wedge L$ on $(x-\delta',x+\delta')$. It follows that $\hat c^L = \hat c$ on $(x-\delta',x+\delta')$. By \eqref{hat c} and \eqref{c_L}, this implies $V_L'= V'$ on $(x-\delta',x+\delta')$, which in turn entails $V_L''= V''$ on $(x-\delta',x+\delta')$. Hence, for any $y\in (x-\delta',x+\delta')$,
\begin{align*}
\beta V(y)  &=  \frac{1}{2} \sigma ^2 y^2V''(y) + (y^\alpha -\mu y)V'(y) + U(\hat c(y)y)-\hat c(y)y V'(y)\\
&=   \frac{1}{2} \sigma ^2 y^2V_L''(y) + (y^\alpha -\mu y)V_L'(y) + U(\hat c^L(y)y)-\hat c^L(y)y V_L'(y) = \beta V_L(y),
\end{align*}
where the first and the last equalities follows from Theorem~\ref{thm:V=V_infty} and Proposition~\ref{prop:properties V_L}. This implies $V=V_L$ on $(x-\delta',x+\delta')$, a contradiction to Proposition~\ref{prop:V>V_L} (ii).
\end{proof}

To concretely illustrate the above results, in the following we focus on the utility function 
\begin{equation}\label{power U}
U(x) := \frac{x^{1-\gamma}}{1-\gamma}\quad \hbox{for}\ x>0,\qquad \hbox{with $0<\gamma <1$}.  
\end{equation}

\begin{lemma}
Assume \eqref{power U}. Then, there exist $C_1, C_2>0$ such that
\begin{equation}\label{V at infty}
C_1 x^{1-\gamma} \le V(x)\quad \forall x>0\qquad \hbox{and}\qquad   V(x)\leq C_2 (1+x^{1-\gamma})\quad \hbox{as $x\to\infty$}.
\end{equation}
In particular, we have
\begin{equation}\label{limsup at infty}
\lim_{x\to\infty}  x^\gamma V'(x)   =\left( \frac{\gamma}{\beta  + \mu (1-\gamma) + \frac{1}{2} \sigma ^2 \gamma(1-\gamma)}\right)^\gamma>0.
\end{equation}
\end{lemma}

\begin{proof}
Consider the constant consumption process $\bar c_t\equiv 1$. For any $x>0$, let $X$ denote the unique strong solution to \eqref{X} with $c = \bar c$. By the definition of $V$ and \eqref{power U}, 
\[
V(x) \ge \frac{1}{1-\gamma}\E\bigg[\int_0^\infty e^{-\beta t} (X^x_t)^{1-\gamma} dt\bigg].
\] 
Recall from Section~\ref{sec:X} that $X_t = (Z_t)^{1/(1-\alpha)}$, with $Z$ explicitly given in \eqref{Z formula}. It follows that
\begin{align*}
V(x) &\ge \frac{1}{1-\gamma}\E\bigg[\int_0^\infty e^{-\beta t} \left(G_t^{-1} \left(x^{1-\alpha} + (1-\alpha)\int_0^t G_s ds \right)\right)^{\frac{1-\gamma}{1-\alpha}} dt \bigg]\\
& \ge \frac{1}{1-\gamma}\E\bigg[\int_0^\infty e^{-\beta t} \left(G_t^{-1} x^{1-\alpha}\right)^{\frac{1-\gamma}{1-\alpha}} dt \bigg]= \frac{x^{1-\gamma}}{1-\gamma} \E\bigg[\int_0^\infty e^{-\beta t} G_t^{\frac{\gamma-1}{1-\alpha}} dt\bigg],
\end{align*}
where $G$ is defined as in \eqref{G}, with $c_t=\bar c_t\equiv 1$, and the second inequality follows from $G_t>0$ for all $t\ge 0$, $1-\alpha>0$, and $\frac{1-\gamma}{1-\alpha}>0$. Noting that the process $G$ is independent of $x$, we conclude from the above inequality that the first part of \eqref{V at infty} holds. 

By Theorem~\ref{thm:V=V_infty} and \eqref{power U}, $V$ satisfies 
\begin{equation}\label{V under U}
\beta V(x)  =  \frac{1}{2} \sigma ^2 x^2V''(x)+ (x^\alpha -\mu x)V'(x) + \frac{\gamma}{1-\gamma} \left(V'(x)\right)^\frac{\gamma-1}{\gamma},\quad \forall x>0. 
\end{equation}
Recall from Theorem~\ref{thm:V=V_infty} that $V'(x)>0$ and $V''(x)\le 0$ for all $x>0$. Also, by the standing assumption $\mu>0$ in \eqref{mu>0}, $x^\alpha -\mu x<0$ for $x>0$ large enough. Hence, \eqref{V under U} implies the existence of $x_0>0$ such that
\[
\beta V(x)  \leq \frac{\gamma}{1-\gamma} \left(V'(x)\right)^\frac{\gamma-1}{\gamma},\quad \hbox{for $x\ge x_0$}. 
\]
Note that $V$ being nonnegative, concave, and nondecreasing entails $ V'(x) \le \frac{V(x)}{x}$ for all $x>0$. The above inequality then yields $\beta x V'(x)\le \frac{\gamma}{1-\gamma} \left(V'(x)\right)^\frac{\gamma-1}{\gamma}$ for $x\ge x_0$, which is equivalent to
\[
V'(x) \leq \left(\frac{\gamma}{\beta(1-\gamma)}\right)^\gamma x^{-\gamma},\quad \hbox{for $x\ge x_0$}. 
\]
Integrating both sides from $x_0$ to $x\ge x_0$ gives 
\[
V(x) \le V(x_0) + \left(\frac{\gamma}{\beta}\right)^\gamma \left(\frac{1}{1-\gamma}\right)^{\gamma+1} (x^{1-\gamma}-x_0^{1-\gamma}),\quad \hbox{for $x\ge x_0$}. 
\]
This shows that the second part of \eqref{V at infty} is true. 

By \eqref{V at infty}, $0<\liminf_{x\to\infty}\frac{V(x)}{x^{1-\gamma}} \le \limsup_{x\to\infty}\frac{V(x)}{x^{1-\gamma}} <\infty$. 
Hence, for any $\{x_n\}_{n\in\N}$ in $(0,\infty)$ such that $x_n\to\infty$ and $\frac{V(x_n)}{x_n^{1-\gamma}}$ converges, we must have $\lim_{n\to\infty}\frac{V(x_n)}{x_n^{1-\gamma}}=c$ for some $0<c<\infty$. Taking $x=x_n$ in \eqref{V under U} and dividing the equation by $x_n^{1-\gamma}$, we get 
\[
\beta \frac{V(x_n)}{x_n^{1-\gamma}}  =  \frac{1}{2} \sigma^2 x_n^{1+\gamma}V''(x_n)+ (x_n^{\alpha-1} -\mu)x_n^\gamma V'(x_n) + \frac{\gamma}{1-\gamma} \left(x_n^\gamma V'(x_n)\right)^\frac{\gamma-1}{\gamma},\quad \forall n\in\N. 
\]
With $c = \lim_{n \to \infty} \frac{V(x_n)}{x_n^{1-\gamma}}$, L'Hospital's rule implies $c(1-\gamma) = \lim_{n \to \infty} x_n^\gamma V'(x_n)$. Using L'Hospital's rule again yields $-c\gamma(1-\gamma) = \lim_{n \to \infty} x_n^{\gamma+1} V''(x_n)$. Thus, as $n\to\infty$, the above equation gives 
\begin{equation*}
\beta c = - \frac{1}{2} \sigma ^2 c\gamma(1-\gamma)-\mu c(1-\gamma) + \frac{\gamma}{1-\gamma}  \left(c(1-\gamma)\right)^\frac{\gamma-1}{\gamma},
\end{equation*}
which has a unique solution 
$
c = \frac{1}{1-\gamma}  \big( \frac{\gamma}{\beta  + \mu (1-\gamma) + \frac{1}{2} \sigma ^2 \gamma(1-\gamma)}\big)^\gamma>0.
$
With $\{x_n\}_{n\in\N}$ arbitrarily chosen, $\lim_{x\to \infty} \frac{V(x)}{x^{1-\gamma}}$ exists and must equal $c$ as above.  L'Hospital's rule then gives the result \eqref{limsup at infty}.
\end{proof}

\begin{proposition}\label{prop:hat c}
Assume \eqref{power U}. Then, 
\[
\lim_{x\to\infty} \hat c(x) =  \frac{\beta}{\gamma}+(1-\gamma)\bigg(\frac{\mu}{\gamma}+\frac{\sigma^2}{2}\bigg). 
\]
Moreover,
\begin{equation}\label{hat c at 0}
\lim_{x\downarrow 0} \hat c(x) =
\begin{cases}
0,\quad &\hbox{if}\ \gamma<\alpha;\\
\left(\beta V(0+)\right)^{-1/\gamma}>0,\quad &\hbox{if}\ \gamma=\alpha;\\
\infty,\quad &\hbox{if}\ \gamma>\alpha.
\end{cases}
\end{equation}
\end{proposition}

\begin{proof}
Under \eqref{power U}, $\hat c(x) = \frac{(V'(x))^{-1/\gamma}}{x}$. It follows that
\[
\lim_{x\to\infty} \hat c(x) = \lim_{x\to\infty} \left(x^\gamma V'(x)\right)^{-\frac{1}{\gamma}}=\frac{\beta}{\gamma}+(1-\gamma)\bigg(\frac{\mu}{\gamma}+\frac{\sigma^2}{2}\bigg),
\]
where the second equality follows from \eqref{limsup at infty}. On the other hand, by Lemma~\ref{lem:V at 0} (ii), 
\[
\lim_{x\downarrow 0} \hat c(x) = \lim_{x\downarrow 0}  \frac{(V'(x))^{-1/\gamma}}{x} = \lim_{x\downarrow 0}  \frac{(\beta V(0+) x^{-\alpha})^{-1/\gamma}}{x} = \left(\beta V(0+)\right)^{-1/\gamma} \lim_{x\downarrow 0} x^{\frac{\alpha}{\gamma}-1},
\]
which directly implies \eqref{hat c at 0}.
\end{proof}

Proposition~\ref{prop:hat c} admits interesting economic interpretation. An agent's consumption behavior is determined by two competing effects, captured by the parameters $\gamma$ and $\alpha$ respectively. First, as in the literature of mathematical finance, $\gamma$ in \eqref{power U} measures the agent's risk aversion: the larger $\gamma$, the stronger the agent's intention to consume capital right away (to get immediate, riskless utility), as opposed to saving capital in the form of $X$, subject to risky, stochastic evolution. On the other hand, $\alpha$ in \eqref{X} measures how efficient capital is used in an economy to produce new capital: the larger $\alpha$, the stronger the upward potential of $X$, and thus the more willing the agent to save capital (i.e. consume less). Now, as in \eqref{hat c at 0}, when capital per capita $X$ dwindles near 0, (i) if risk aversion of the agent is not so strong relative to the efficiency of capital production (i.e. $\gamma<\alpha$), the effect of $\alpha$ prevails, so that the agent (in the limit) saves all capital to fully exploit the upward potential of $X$; (ii) if risk aversion of the agent is very strong relative to the efficiency of capital production (i.e. $\gamma>\alpha$), the effect of $\gamma$ prevails, so that the agent consumes capital as fast as possible, to reduce risky position in $X$; (iii) if risk aversion of the agent is comparable to the efficiency of capital production (i.e. $\gamma=\alpha$), the effects of $\alpha$ and $\gamma$ are balanced, leading to bounded, positive consumption of the agent. 

\begin{corollary}\label{coro:power U results}
Assume \eqref{power U}. If $\gamma\le \alpha$, as long as $L>0$ is large enough, $V_L(x) = V(x)$ for all $x>0$. If $\gamma>\alpha$, then for any $L>0$, $V_L(x)< V(x) $ for all $x>0$; moreover, for $L>\frac{\beta}{\gamma}+(1-\gamma)\big(\frac{\mu}{\gamma}+\frac{\sigma^2}{2}\big)$, there exists $x^*>0$ such that $\hat c^L(x^*) \neq \hat c(x^*) \wedge L$.
\end{corollary}

\begin{proof}
Since $\hat c(x)$ is by definition continuous on $(0,\infty)$, Proposition~\ref{prop:hat c} implies (i) $\hat c(x)$ is bounded on $(0,\infty)$ if and only if $\gamma\le \alpha$, and (ii) $\inf_{x>0}\hat c(x)\le \frac{\beta}{\gamma}+(1-\gamma)\big(\frac{\mu}{\gamma}+\frac{\sigma^2}{2}\big)$. The result then follows from Proposition~\ref{prop:V>V_L} and Corollary~\ref{coro:statement not true}. 
\end{proof} 

The next two results focus on the specific case $\gamma=\alpha$. The purpose is twofold. First, we demonstrate that the value function $V$ and optimal consumption $\hat c$ can be solved explicitly. Second, as we will see, $\hat c$ is constant (and thus bounded), so that Corollary~\ref{coro:statement not true} is inconclusive on the failure of ``$\hat c^L = \hat c\wedge L$''. Explicit calculation shows that ``$\hat c^L = \hat c\wedge L$'' holds for some, but not all, $L>0$.  

\begin{proposition}\label{prop:explicit case}
Assume \eqref{power U} with $\gamma= \alpha$. Then, $V(x)=\zeta\cdot \left(\frac{x^{1-\alpha}}{1-\alpha} + \frac{1}{\beta}\right)$, with 
\begin{align}\label{zeta}
\zeta &:= \left( \frac{\alpha}{\beta +\mu (1-\alpha)+ \frac{1}{2} \sigma ^2\alpha  (1-\alpha)}\right)^\alpha>0. 
\end{align}
Moreover, the optimal consumption \eqref{hat c} is a constant process given by 
\begin{equation}\label{constant c}
\hat c_t \equiv \frac{\beta}{\alpha}+(1-\alpha)\bigg(\frac{\mu}{\alpha}+\frac{\sigma^2}{2}\bigg)>0.
\end{equation}
\end{proposition}

\begin{proof}
By Theorem~\ref{thm:V=V_infty}, \eqref{power U}, and $\gamma=\alpha$, $V$ is a classical solution to 
\begin{equation*}
\beta v(x)  =  \frac{1}{2} \sigma ^2 x^2v''(x)+ (x^\alpha -\mu x)v'(x) + \frac{\alpha}{1-\alpha} \left(v'(x)\right)^\frac{\alpha-1}{\alpha},\quad \forall x>0. 
\end{equation*}
We plug the ansatz $v(x) = a x^{1-\alpha} + b$, for some $a, b\in \R$, in the above equation. Equating the $x^{1-\alpha}$ terms on both sides leads to 
\[
\beta a= \frac{1}{2} \sigma ^2 a(1-\alpha)(-\alpha)-\mu a(1-\alpha)+ \frac{\alpha}{(1-\alpha)^\frac{1}{\alpha}} {a^\frac{\alpha-1}{\alpha}},
\]
which implies $a = \frac{\zeta}{1-\alpha}$, with $\zeta$ as in \eqref{zeta}. 
Similarly, equating the constant terms on both sides yields $\beta b = a(1-\alpha)$, which implies $b=\frac{\zeta}{\beta}$, with $\zeta$ as in \eqref{zeta}.
By construction, $v(x) = \zeta\cdot \big(\frac{x^{1-\alpha}}{1-\alpha} + \frac{1}{\beta}\big)$ is nonnegative, concave, strictly increasing, and satisfies the linear growth condition \eqref{linear growth}. Moreover, 
\[
\lim_{x \downarrow 0} \frac{(U')^{-1} (v'(x))}{x^\alpha} 
= \lim_{x \downarrow 0} \frac{(\zeta x^{-\alpha})^{-1/\alpha}}{x^\alpha}=\lim_{x \downarrow 0} \zeta^{-1/\alpha}  x^{1-\alpha}= 0,
\]
i.e. \eqref{C*/x^alpha'} is satisfied. Hence, we conclude from Theorem~\ref{thm:main} that $V(x)=v(x)$ for all $x>0$, and the optimal consumption process $\hat c$ is given by
\[
\hat c_t = \frac{(U')^{-1} (v'(X_t))}{X_t} = \frac{(\zeta X_t^{-\alpha})^{-1/\alpha}}{X_t} = \zeta^{-1/\alpha}=\frac{\beta}{\alpha}+(1-\alpha)\bigg(\frac{\mu}{\alpha}+\frac{\sigma^2}{2}\bigg),\quad \forall t\ge 0.  
\]
\end{proof}

The constant consumption \eqref{constant c} turns out to be the threshold, uniform in $x>0$, for ``$\hat c^L= \hat c \wedge L$'' to hold. 

\begin{proposition}
Assume \eqref{power U} with $\gamma= \alpha$.  Then, $\hat c^L(x) = \hat c (x) \wedge L$ for all $x>0$ if and only if $L\ge \frac{\beta}{\alpha}+(1-\alpha)\big(\frac{\mu}{\alpha}+\frac{\sigma^2}{2}\big)$. 
\end{proposition}

\begin{proof}
By Proposition~\ref{prop:explicit case}, $\hat c (x)\equiv \frac{\beta}{\alpha}+(1-\alpha)\big(\frac{\mu}{\alpha}+\frac{\sigma^2}{2}\big)$. If $L\ge \frac{\beta}{\alpha}+(1-\alpha)\big(\frac{\mu}{\alpha}+\frac{\sigma^2}{2}\big)$, by Proposition~\ref{prop:V>V_L} we have $V_L = V$ on $(0,\infty)$, which in turn implies $\hat c^L = \hat c = \hat c\wedge L$ on $(0,\infty)$.  On the other hand, if $\hat c^L(x) = \hat c (x) \wedge L$ for all $x>0$, assume to the contrary that $L< \frac{\beta}{\alpha}+(1-\alpha)\big(\frac{\mu}{\alpha}+\frac{\sigma^2}{2}\big)$. Then, $\hat c^L(x) = L$ for all $x>0$. By Proposition~\ref{prop:properties V_L}, $V_L$ is a classical solution to 
\[
\beta v(x)  =  \frac{1}{2} \sigma ^2 x^2 v''(x) + (x^\alpha -\mu x)v'(x) + \frac{(Lx)^{1-\alpha}}{1-\alpha} - Lx v'(x)\quad \forall x>0. 
\]
Take the ansatz $v(x) = a x^{1-\alpha} + b$ for some $a, b\in \R$. By the same argument in Proposition~\ref{prop:explicit case}, we obtain $v(x)=\zeta_L\cdot \left(\frac{x^{1-\alpha}}{1-\alpha} + \frac{1}{\beta}\right)$, where
\[
\zeta_L := \frac{L^{1-\alpha}}{\beta+ (1-\alpha)\big(\mu+L+\frac{\alpha\sigma ^2 }{2} \big)}
\]
By construction, $v(x) = \zeta_L\cdot \big(\frac{x^{1-\alpha}}{1-\alpha} + \frac{1}{\beta}\big)$ is nonnegative, concave, strictly increasing, and satisfies the linear growth condition \eqref{linear growth}. In view of the characterization of $V_L$ in Remark \ref{rem:characterize V_L}, we have $V_L(x) = v(x)$ for all $x>0$. Now, for any $x>0$,
\begin{align*}
\frac{(U')^{-1}(V_{L}'(x))}{x} &= \frac{(\zeta_L x^{-\alpha})^{\frac{-1}{\alpha}}}{x} = \zeta_L ^{\frac{-1}{\alpha}}\le  \left(\frac{L^{1-\alpha}}{ L(1-\alpha)}\right)^{\frac{-1}{\alpha}} = L (1-\alpha)^{\frac{1}{\alpha}} < L,
\end{align*}
which implies $\hat c^L(x) = \min\big\{\frac{(U')^{-1}(V_{L}'(x))}{x} , L\big\} = \frac{(U')^{-1}(V_{L}'(x))}{x}< L$, a contradiction to $\hat c^L(x) = L$. 
\end{proof}


\appendix

\section{Derivation of Proposition~\ref{prop:properties V_L}}\label{sec:appendix}

In this appendix, we will establish Proposition~\ref{prop:properties V_L} by generalizing arguments in \cite{MZ08} to infinite horizon. As mentioned in Section~\ref{sec:V_L}, \cite{MZ08} studies a similar problem to $V_L$ in \eqref{V_L}, yet under finite horizon and with the specific bound $L=1$. As we will see, many arguments in \cite{MZ08} can be modified without much difficulty to infinite horizon. A distinctive exception is the derivation of the dynamic programming principle for $V_L$; see Lemma~\ref{lem:viscosity V_L} below for details.  

\begin{lemma}\label{lem:bound V_L}
\begin{itemize}
\item [(i)] For any $L>0$, $V_L$ is concave on $(0,\infty)$. 
\item [(ii)] There exists $\varphi_0>0$ such that $V_L(x) \le x+ \varphi_0$ for all $x>0$ and $L>0$.
\end{itemize}
\end{lemma}

\begin{proof}
(i) This follows from the same argument in \cite[Theorem 5.1]{MZ08}.

(ii) We will prove this result by modifying the argument in the first part of \cite[Lemma 3.2]{MZ08}. Define $\varphi(x) := x +\varphi_0$ with $\varphi_0>0$ to be determined later. Fix $L>0$. For any $c\in\C_L$, $x>0$, and $T>0$, It\^{o}'s formula implies
\begin{equation}\label{la}
0 \leq \E[e^{-\beta T} \varphi(X^x_{T})] = \varphi(x) + \E\bigg[\int_{0}^{T} e^{ -\beta s} (-\beta \varphi(X^x_s) + (X^x_s)^{\alpha} - \mu X^x_s - c_s X^x_s) ds \bigg]. 
\end{equation}
Note that $-\E[\int_{0}^{T}e^{-\beta s}\sigma X_s \ dW_s]$ disappears from the above inequality because $\int_{0}^{\cdot} e^{-\beta s}\sigma X_s dW_s$ is a martingale, thanks to the second part of \eqref{eq2.5}. By \eqref{U2} and $\mu>0$, we have $\sup_{y\ge 0}\{U(y)-y\}<\infty$ and $A:= \sup_{x\ge 0}\{x^\alpha - \mu x\}<\infty$. We can therefore take $\varphi_0>0$ large enough such that 
\begin{equation}\label{choice}
-\beta \varphi(x) + (x^{\alpha} - \mu x) +  \sup_{y\ge 0}\{U(y)-y\} \leq -\beta \varphi_0 + A \ +  \sup_{y\ge 0}\{U(y)-y\} < 0,\quad x\ge 0.
\end{equation}  
This, together with \eqref{la}, yields
\begin{align*} 
0 \leq \E[e^{-\beta T} \varphi(X^x_{T})] 
&\leq \varphi(x) - \E\bigg[\int_{0}^{T} e^{ -\beta s}U(c_s X^x_s) ds\bigg],
\end{align*}
Hence, by using Fatou's lemma as $T\to\infty$ and then taking supremum over $c\in\C_L$, we get the desired result $V_L(x)\le \varphi(x)$. Finally, note that our choice of $\varphi_0>0$ can be made independent of both $L>0$ and $x>0$. Indeed, the right hand side of \eqref{choice}, which involves $\varphi_0$, does not depend on either $L$ or $x$. 
\end{proof}

Next, we derive the dynamic programming principle for $V_L$, to show that it is a viscosity solution. As explained in detail under \eqref{DPP'}, arguments in \cite{MZ08} only lead us to a {\it weak} dynamic programming principle. Additional probabilistic arguments are invoked to upgrade this weak principle. 

\begin{lemma}\label{lem:viscosity V_L}
For any $L>0$, $V_L$ is a continuous viscosity solution to \eqref{HJB V_L}.
\end{lemma}

\begin{proof}
Fix $L>0$. The continuity of $V_L$ on $(0,\infty)$ is a direct consequence of Lemma~\ref{lem:bound V_L} (i). In view of \cite[Chapter V]{FS-book-06} and \cite[Chapter 4]{Pham-book-09}, to prove the viscosity solution property, it suffices to show the following dynamic programming principle: for any $x>0$,
\begin{equation*}
V_L(x) = \sup_{c\in\C_L} \E\left[\int_{0}^{\tau} e^{-\beta t} U(c_t X^x_t) dt + e^{-\beta \tau }V_L(X^x_\tau)\right],\quad \forall \tau\in \T, 
\end{equation*}
where $\T$ denotes the set of all stopping times. The ``$\le$'' relation is straightforward to derive. Indeed, given $c\in\C_L$, we have, for any $\tau\in\T$, that 
\begin{align}
\E\left[\int_{0}^{\infty} e^{-\beta t} U(c_t X^x_t) dt \right] &= \E\left[\int_{0}^{\tau} e^{-\beta t} U(c_t X^x_t) dt + e^{-\beta \tau }\E\left[\int_{\tau}^{\infty} e^{-\beta (t-\tau)} U(c_t X^x_t) dt\ \middle|\ \F_\tau\right] \right]\notag\\
&= \E\left[\int_{0}^{\tau} e^{-\beta t} U(c_t X^x_t) dt + e^{-\beta \tau }\E\left[\int_{\tau(\omega)}^{\infty} e^{-\beta (t-\tau(\omega))} U(c^{\tau,\omega}_{t-\tau} X^{X^x_\tau(\omega)}_{t-\tau}) dt\right] \right]\notag\\
&= \E\left[\int_{0}^{\tau} e^{-\beta t} U(c_t X^x_t) dt + e^{-\beta \tau }\E\left[\int_{0}^{\infty} e^{-\beta t} U(c^{\tau,\omega}_t X^{X^x_\tau(\omega)}_t) dt\right] \right]\notag\\
&\le \E\left[\int_{0}^{\tau} e^{-\beta t} U(c_t X^x_t) dt + e^{-\beta \tau }V_L(X^x_\tau)\right].\label{DPP <}
\end{align}
Here, the second line follows from \cite[Proposition A.1]{BH13-game}, with $c^{\tau,\omega}\in\C_L$ defined by $c_s^{\tau,\omega}(\bar\omega):= c_{\tau(\omega)+s}(\omega\otimes_{\tau(\omega)}\bar\omega)$, $s\ge 0$, for each fixed $\omega\in\Omega$; recall \eqref{concatenate}. The third line, on the other hand, follows from the definition of $V_L$. Now, taking supremum over $c\in\C_L$ gives the desired ``$\le$'' relation. 

The rest of the proof focuses on deriving the converse inequality
\begin{equation}\label{DPP}
V_L(x) \ge \sup_{c\in\C_L} \E\left[\int_{0}^{\tau} e^{-\beta t} U(c_t X^x_t) dt + e^{-\beta \tau }V_L(X^x_\tau)\right],\quad \forall \tau\in \T. 
\end{equation}
Following the arguments in \cite[Theorem 3.3]{MZ08} and using the estimates in \eqref{eq2.5} and \eqref{eq2.7}, we can derive a weaker version of \eqref{DPP}:
\begin{equation}\label{DPP'}
V_L(x) \ge \sup_{c\in\C_L} \E\left[\int_{0}^{r} e^{-\beta t} U(c_t X^x_t) dt + e^{-\beta r }V_L(X^x_r)\right],\quad \forall r\ge 0. 
\end{equation}
Note that the arguments in \cite[Theorem 3.3]{MZ08} directly give the stronger statement \eqref{DPP} under finite horizon $T>0$, with $\T$ replaced by $\T_T$, the set of stopping times taking values in $[0,T]$ a.s. The same arguments, however, only render the weaker statement \eqref{DPP'} under infinite horizon. This is because with finite horizon $T>0$, one can derive an estimate for $\E[\sup_{0\le t\le T} X^2_t]$, i.e. (2.7) in \cite{MZ08}, which ensures that (3.14) in \cite{MZ08} holds {\it simultaneously} for all $\tau\in\T_T$. When the time horizon is infinite, one would need a corresponding estimate for $\E[\sup_{0\le t<\infty} X^2_t]$, which is often unavailable. In our case, we only have the estimates \eqref{eq2.5} and \eqref{eq2.7}, which ensure that (3.14) in \cite{MZ08} holds only for {\it each} deterministic time $r\ge 0$.  

In the following, we will show that the weaker statement \eqref{DPP'} in fact implies \eqref{DPP}. 
First, we claim that for any $c\in\C_L$ and $x>0$, the process $\int_0^t e^{-\beta s} U(c_s X^x_s) ds + e^{-\beta t}V_L(X^x_t)$, $t\ge 0$, is a supermartingale. Given $0\le r\le t$, it holds for a.e. $\omega\in\Omega$ that  
\begin{align*}
&\E\left[\int_0^t e^{-\beta s} U(c_s X^x_s) ds + e^{-\beta t}V_L(X^x_t)\ \middle|\ \F_r\right](\omega)\\
 = &\int_0^r e^{-\beta s} U(c_s X^x_s) ds(\omega) + e^{-\beta r}\E\left[\int_r^t e^{-\beta (s-r)} U(c_s X^x_s) ds + e^{-\beta (t-r)}V_L(X^x_t)\ \middle|\ \F_r\right] (\omega) \\
 = &\int_0^r e^{-\beta s} U(c_s X^x_s) ds(\omega) + e^{-\beta r}\E\left[\int_r^t e^{-\beta (s-r)} U(c^{r,\omega}_{s-r} X^{X^x_r(\omega)}_{s-r}) ds + e^{-\beta (t-r)}V_L\left(X^{X^x_r(\omega)}_{t-r}\right)\right]\\ 
 = &\int_0^r e^{-\beta s} U(c_s X^x_s) ds(\omega) + e^{-\beta r}\E\left[\int_0^{t-r} e^{-\beta s} U\left(c^{r,\omega}_s X^{X^x_r(\omega)}_s\right) ds + e^{-\beta (t-r)}V_L\left(X^{X^x_r(\omega)}_{t-r}\right)\right],
\end{align*}
where the third line follows from \cite[Proposition A.1]{BH13-game}, with $c^{r,\omega}\in\C_L$ defined by $c_s^{r,\omega}(\bar\omega):= c_{r+s}(\omega\otimes_{r}\bar\omega)$, $s\ge 0$, for each fixed $\omega\in\Omega$; recall \eqref{concatenate}. This, together with \eqref{DPP'}, yields
\[
\E\left[\int_0^t e^{-\beta s} U(c_s X^x_s) ds + e^{-\beta t}V_L(X^x_t)\ \middle|\ \F_r\right] \le \int_0^r e^{-\beta s} U(c_s X^x_s) ds + e^{-\beta r} V_L(X^x_r)\quad \hbox{a.s.}
\]
This shows the desired supermartingale property. Now, for any $x>0$ and $\tau\in\T$, by the optional sampling theorem, 
\[
V_L(x)\ge \E\left[\int_0^{\tau\wedge T} e^{-\beta s} U(c_s X^x_s) ds + e^{-\beta (\tau\wedge T)}V_L(X^x_{\tau\wedge T})\right],\quad \forall\ T>0.
\]
As $T\to\infty$, thanks to Fatou's lemma and the continuity of $V_L$, we obtain \eqref{DPP}. 
\end{proof}

The next comparison result follows directly from the argument in \cite[Theorem 4.1]{MZ08}. The argument is in fact slightly simpler here, as the time variable is not involved in our infinite-horizon setup; see also a very similar proof in \cite[Proposition 4.1]{Liu14} for a related infinite-horizon problem. 

\begin{lemma}\label{lem:comparison}
Fix $L>0$. For any $0<a<b$, if $u_1, u_2 \in C([a,b])$ are two viscosity solutions to 
\begin{equation}\label{BV V_L}
\beta v(x)  =  \frac{1}{2} \sigma ^2 x^2v''( x) + (x^\alpha -\mu x)v'(x) + \tilde U_L (x, v'(x))  \quad \text{for}\ x\in(a,b),
\end{equation}
with $u_1(a)=u_2(a)$ and $u_1(b)=u_2(b)$, then $u_1 \equiv u_2$. 
\end{lemma}

Now, we are ready to prove Proposition~\ref{prop:properties V_L}.

\begin{proof}[Proof of Proposition~\ref{prop:properties V_L}]
In view of Lemma~\ref{lem:bound V_L}, it remains to show that $V_L$ belongs to $C^2((0,\infty))$ and solves \eqref{HJB V_L}. 
For any $0<a<b$, consider the boundary value problem \eqref{BV V_L} with $v(a)=V_L(a)$ and $v(b) =  V_L(b)$. 
Thanks to the boundedness of $c\in\C_L$, the same estimate for $|\tilde U_L (x_1, p_1)-\tilde U_L (x_2, p_2)|$ in \cite[Theorem 4.2]{MZ08} still holds, which means that the condition (5.18) in \cite{Morimoto-book-10} is true under current setting. We then conclude from \cite[Theorem 5.3.7]{Morimoto-book-10} that there exists a classical solution $v\in C^2((a,b))\cap C([a,b])$ to \eqref{BV V_L}. Since $v$ is also a viscosity solution, Lemmas~\ref{lem:viscosity V_L} and \ref{lem:comparison} imply that $V_L = v$ on $[a,b]$, and thus $V_L\in C^2((a,b))$. With $0<a<b$ arbitrarily chosen, we have $V_L\in C^2((0,\infty))$ and solves \eqref{HJB V_L} in the classical sense. 
\end{proof}


\end{document}